\documentclass[12pt]{amsart}

\usepackage{graphicx,amsmath,mathtools,amssymb}
\usepackage{epsfig,color,hyperref,fancyhdr,geometry}
\usepackage{dsfont}

\newtheorem{theorem}{Theorem}[section]
\newtheorem{lemma}[theorem]{Lemma}
\newtheorem{definition}[theorem]{Definition}

\newtheorem{proposition}[theorem]{Proposition}
\newtheorem{remark}[theorem]{Remark}
\newtheorem{corollary}[theorem]{Corollary}
\newtheorem{conjecture}[theorem]{Conjecture}
\newtheorem{problem}[theorem]{Problem}

\numberwithin{equation}{section}
\numberwithin{figure}{section}

\begin{document}

\title{Search for torsion in Khovanov homology}

\author{Sujoy Mukherjee}

\author{J\'{o}zef H. Przytycki}

\author{Marithania Silvero}

\author{Xiao Wang}

\author{Seung Yeop Yang}

\begin{abstract}
	
	In the Khovanov homology of links, presence of $\mathbb{Z}_2$-torsion is a very common phenomenon. Finite number of examples of knots with $\mathbb{Z}_n$-torsion for $n>2$ were also known, none for $n>8$. In this paper, we prove that there are infinite families of links whose Khovanov homology contains $\mathbb{Z}_n$-torsion for $2 < n < 9$ and $\mathbb{Z}_{2^s}$-torsion for $s < 24$. We also introduce $4$-braid links with $\mathbb{Z}_3$-torsion which are counterexamples to the PS braid conjecture. We also provide an infinite family of knots with $\mathbb{Z}_5$-torsion in reduced Khovanov homology and $\mathbb{Z}_3$-torsion in odd Khovanov homology.

\end{abstract}



\maketitle

\tableofcontents

\section{Introduction}

Khovanov homology  \cite{Kho1}, a categorification of the Jones polynomial \cite{Jon}, has been computed for many links and the experimental data suggests that there is an abundance of $\mathbb{Z}_2$-torsion.\footnote{Reasons for abundance of $\mathbb{Z}_2$ were discussed in \cite{AP,PPS,Prz2,PrSa,Shu3}.} However, other torsion groups seem to appear rarely \cite{Kho2,Shu3}. The smallest known knot with $\mathbb{Z}_4$-torsion was the torus knot $T(4,5)$ with crossing number $15$\footnote{K. Murasugi proved that the crossing number of the torus link $T(m,n)$, $m \leq n$ is equal to $(m-1)n$ \cite{Mur}.}, while the smallest known knot with $\mathbb{Z}_3$ and $\mathbb{Z}_5$-torsion was the torus knot $T(5,6)$ with crossing number $24$. Despite not having been proven yet, calculations by D. Bar-Natan, A. Shumakovitch, and L. Lewark suggest $\mathbb{Z}_{p^k}$-torsion in the torus knot $T(p^k,p^k+1)$ for $p^k >3$ \cite{PrSa}. In this paper, we show the existence of $\mathbb{Z}_n$-torsion for $2 < n < 9$ and $2^s$-torsion for $s \leq 23$ in the Khovanov homology of some infinite family of knots and links of two components with braid index $4$. Moreover, we found a link of two components with braid index $4$ and $\mathbb{Z}_3$-torsion in Khovanov homology.

The paper is organized as follows. In Section \ref{2}, we review the main ideas of Khovanov homology and standardize our notations. The main families of links we analyze are twist deformations of torus links.\footnote{Khovanov homology of torus links was studied in \cite{BFLZ,GOR,IW,Roz,Sto1,Sto2,Tur,Wil}.} In particular, we discuss the long exact sequence of Khovanov homology \cite{AP,Kho1,Sto2,Vir1,Vir2}. Section \ref{aa} is devoted to our main results. In Subsection \ref{3.1}, we analyze the family of torus links $T(m,m+2)$ and their deformations by twists. In particular, we show that these families have infinite sub-families with $\mathbb{Z}_3, \mathbb{Z}_4, \mathbb{Z}_5, \mathbb{Z}_7,$ and $\mathbb{Z}_8$-torsion in Khovanov homology. We explicitly use the fact that the links in these families have smoothing number one. In the next subsection, we study Khovanov homology of the twist deformations of the torus knots of type $T(m,m+1)$. In Subsection \ref{ss 3.3},  we discuss examples of knots and links with $\mathbb{Z}_3, \mathbb{Z}_4, \mathbb{Z}_5, \mathbb{Z}_7,$ and $\mathbb{Z}_8$-torsion. In many cases, we give examples smaller than known ones. In particular, we give an example of a knot diagram of $22$ crossings with $\mathbb{Z}_3$-torsion in Khovanov homology. In Subsection \ref{3.2}, we give examples of infinite families of knots with braid index $4$ having $\mathbb{Z}_{2^{s}}$-torsion for $2 \leq s \leq 23$.\footnote{Due to our computational limitations we could only compute up to $s=23.$ See Corollary \ref{Corollary 3.14} and Conjecture \ref{bigtorsions}.} These knots are counterexamples to parts of a conjecture in \cite{PrSa}.

In the penultimate section, we discuss an infinite family of knots with $\mathbb{Z}_5$-torsion in reduced Khovanov homology. Torsion in reduced Khovanov homology is rarer than torsion in unreduced Khovanov homology. Our smallest example with $\mathbb{Z}_5$-torsion is $T^{(-8)}(7,9)$ having a knot diagram with $46$ crossings. The previously known smallest example with $\mathbb{Z}_5$-torsion was $T(7,8)$ with $48$ crossings. We also give examples of infinite families of knots with $\mathbb{Z}_3$-torsion in odd Khovanov homology. Finally, the Appendix consists of Khovanov homology tables of some links relevant to this paper.

This paper is the third in the series of papers resulting from Mathathon, a mathematical marathon run at the George Washington University every December \cite{CMPWY1,CMPWY2}.\footnote{The first Mathathon took place between December 14 and December 22, 2015.}

\section{Preliminaries}\label{2}

In this section we review the main ideas of Khovanov homology of unoriented framed links following O. Viro  \cite{Vir1}.

Let $D$ be an unoriented diagram of a link, and $cr(D)$ the set of its crossings. A Kauffman state $s$ assigns a label, $A$ or $B$, to each crossing of $D$ (i.e., $s: \, cr(D) \to \{A,B\}$). Let $\mathcal{S}$ be the collection of $2^{|cr(D)|}$ states of $D$. For $s \in \mathcal{S}$, write $\sigma = \sigma(s) = |s^{-1}(A)| - |s^{-1}(B)|$.
We denote by $sD$ the result of smoothing each crossing according to the convention in Figure \ref{markers}. Let $|sD|$ be the number of circles in $sD$. Enhance a state $s$ with a map $e$ which associates a sign $\epsilon_i = \pm 1$ to each of the $|sD|$ circles in $sD$. We keep the letter $s$ for the enhanced state $(s,e)$ to avoid cumbersome notation. Write $\tau = \tau(s) = \sum_{i=1}^{|sD|} \epsilon_i$, and define, for the enhanced state $s$, the integers
$$
a = a(s) = \sigma, \quad b = b(s) = \sigma + 2\tau.
$$

\begin{figure}
\centering
\includegraphics[width = 14cm]{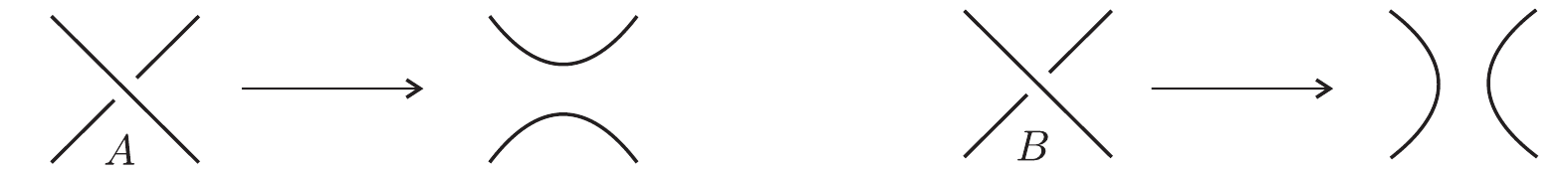}
\caption{\small{Smoothing of a crossing according to its $A$ or $B$-label.}}
\label{markers}
\end{figure}

Given $s$ and $s'$, two enhanced states of $D$, $s'$ is said to be adjacent to $s$ if $a(s') = a(s) - 2$, $b(s') = b(s)$, both of them have identical labels except for one changed crossing, where $s$ assigns an $A$-label and $s'$ a $B$-label, and they assign the same sign to the common circles in $sD$ and $s'D$.

Let $C_{a,b}(D)$ be the free abelian group generated by the set of enhanced states $s$ of $D$ with $a(s) = a$ and $b(s) = b$. Order the crossings in $D$. Now, fix an integer $b$ and consider the descendant complex
$$
\cdots \quad \longrightarrow \quad C_{a,b}(D) \quad \stackrel{\partial_{a,b}}{\longrightarrow} \quad C_{a-2,b}(D) \quad \longrightarrow \quad \cdots
$$
with differential $\partial_{a,b}(s) = \sum\limits_{s' \in \mathcal{S}} (s:s') s'$, with $(s:s') = 0$ if $s'$ is not adjacent to $s$ and otherwise $(s:s') = (-1)^t$, with $t$ the number of $B$-labeled crossings in $s$ coming after the changed crossing.

\begin{theorem} {\rm{\cite{Vir1,Vir2}}}
With the notation above, $\partial_{a-2,b} \circ \partial_{a,b} = 0$ and the corresponding homology groups
$$
H_{a,b}(D)=\frac{\textnormal{ker} (\partial_{a,b})}{\textnormal{im}(\partial_{a+2,b})}
$$
are invariant under Reidemeister II and III moves, and therefore they are invariants of framed unoriented links. These groups categorify the unreduced Kauffman bracket polynomial.
\end{theorem}

We will refer to the groups $H_{a,b}(D)$ as the {\it framed Khovanov homology groups} of $D$.

Following \cite{Vir1}, the skein relation associated to the Kauffman bracket can be categorified in the following way:

Let $D_A$ (respectively, $D_B$) be the link diagram obtained after smoothing $D$ at a crossing $v$ of $D$ according to an $A$-label (respectively, $B$-label). Note that we do not keep track of the vertex $v$ in the notation, as it will not be relevant in our settings. Consider the map
$$
\alpha: \, C_{a+1,b+1}(D_B) \, \longrightarrow \, C_{a,b}(D),
$$
which sends an enhanced state $s$ of $D_B$ to the enhanced state of $D$ assigning a $B$-label to the crossing $v$ and keeping for the rest of the crossings the same labels as those assigned by $s$. The signs of the circles are also preserved.

Now consider the map
$$
\beta: \, C_{a,b}(D) \, \longrightarrow \, C_{a-1,b-1}(D_A),
$$
which sends each enhanced state with a $B$-label at the crossing $v$ to 0 and each enhanced state with an $A$-label at $v$ to the enhanced state of $D_A$ assigning the same labels and signs of the circles.

A short exact sequence of complexes can be obtained by combining both homomorphisms:
$$
0 \quad \longrightarrow \quad C_{*+1,*+1}(D_B) \quad \stackrel{\alpha} {\longrightarrow} \quad C_{*,*}(D) \quad \stackrel{\beta} {\longrightarrow} \quad C_{*-1,*-1}(D_A) \quad \longrightarrow \quad 0,
$$
which gives rise to the following long exact sequence of homology,
$$
\begin{array}{ccccccc}
\cdots& \stackrel{\gamma_*} {\longrightarrow} & H_{a+3,b+1}(D_B) & \stackrel{\alpha_*} {\longrightarrow} & H_{a+2,b}(D) &
\stackrel{\beta_*} {\longrightarrow} & H_{a+1,b-1}(D_A) \\
& \stackrel{\gamma_*} {\longrightarrow} & H_{a+1,b+1}(D_B) & \stackrel{\alpha_*} {\longrightarrow} & H_{a,b}(D) &
\stackrel{\beta_*} {\longrightarrow} & H_{a-1,b-1}(D_A) \\
& \stackrel{\gamma_*} {\longrightarrow} & H_{a-1,b+1}(D_B) & \stackrel{\alpha_*} {\longrightarrow} &\cdots.
\end{array}
$$

Let $\vec D$ be the diagram obtained after fixing an orientation for $D$ and let $w = w(\vec D)$ be its writhe. Then, the classical Khovanov (co)homology of $\vec D$ can be obtained from the framed Khovanov homology in the following way
$$H^{i,j}(\vec D) = H_{w-2i,3w-2j}(D),$$
with these groups being link invariants and categorifying the unreduced Jones polynomial. Note that the indices $i$ and $j$ agree with those used in most tabulations (see, for example, the tables at The Knot Atlas \cite{KA}).

In the opposite direction, $$H_{a,b}(D)= H^{\frac{w-a}{2},\frac{3w-b}{2}}(\vec D).$$

Since the use of subscripts/superscripts makes it clear whether we refer to classical Khovanov (co)homology or framed Khovanov homology, from now on we will write $D$ for both, the oriented and unoriented diagrams, unless otherwise stated.

\begin{lemma}\label{Corollary 2.2.}
Let us assume that the smoothed crossing $v$ is positive. If we orient $D_A$ in the only way preserving the orientation of $D$, and choose an orientation of $D_B$ that keeps the orientation of the components not involved in the crossing $v$, then the previous long exact sequence becomes
$$
\begin{array}{ccccccc}
\cdots& \stackrel{\gamma_*} {\longrightarrow} & H^{\frac{w(D_B)-w(D)-3}{2}+i,\frac{3(w(D_B)-w(D))-1}{2}+j}(D_B) & \stackrel{\alpha_*} {\longrightarrow} & H^{i-1,j}(D) &
\stackrel{\beta_*} {\longrightarrow} & H^{i-1,j-1}(D_A) \\
& \stackrel{\gamma_*} {\longrightarrow} & H^{\frac{w(D_B)-w(D)-1}{2}+i,\frac{3(w(D_B)-w(D))-1}{2}+j}(D_B) & \stackrel{\alpha_*} {\longrightarrow} & H^{i,j}(D) &
\stackrel{\beta_*} {\longrightarrow} & H^{i,j-1}(D_A) \\
& \stackrel{\gamma_*} {\longrightarrow} & H^{\frac{w(D_B)-w(D)+1}{2}+i,\frac{3(w(D_B)-w(D))-1}{2}+j}(D_B) & \stackrel{\alpha_*} {\longrightarrow} &\cdots.
\end{array}
$$
\end{lemma}

\begin{remark}
	Let $v$ be a positive crossing of a diagram $D'$. Perform a $t_{k}$-move after the crossing $v$. See Figure \ref{tkmove}, where $k$ is positive. However, we also allow $k$ to be non-positive. Consider the long exact sequence in Lemma \ref{Corollary 2.2.} for $D=t_{k}(D^{'})$ and $D_{A}=t_{k-1}(D^{'})$ (see \cite{Prz1}).  Observe that $D_{B},$ up to framing, in the long exact sequence does not depend on $k$.  This observation is used throughout the paper.
\end{remark}

\begin{figure}[h!]
\centering
\includegraphics[width = 14cm]{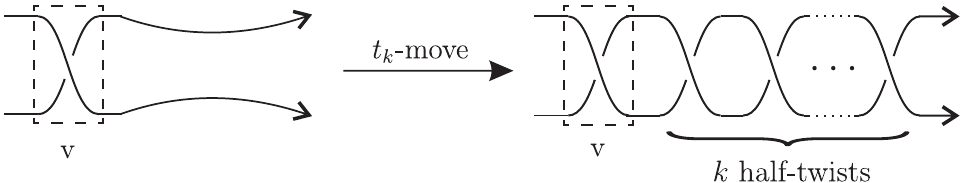}
\caption{\small{A $t_k$-move.}}
\label{tkmove}
\end{figure}

\begin{definition}
The $A$-smoothing number of a diagram $D$ is the minimal number of $A$-smoothings needed in order to transform $D$ into a trivial link (compare \cite{Jab}). In particular, if the $A$-smoothing number of $D$ equals one, then there exists a crossing such that $D_A$ is a trivial link. The $A$-smoothing number of a link $L$ is the minimum over all diagrams of $L$. A similar definition holds for the case of $B$-smoothing.
\end{definition}

\section{Torsion in the Khovanov homology of $T^{(k)}(m,n)$}\label{aa}

Write $T(m,n)$ for a torus link of type $(m,n).$ Any torus link $T(m,n)$ with $m \geq 2$ can be represented as the closed braid on $m$ strands $\widehat{\beta}_{m,n} = (\sigma_{m-1} \sigma_{m-2} \cdots \sigma_2 \sigma_1)^n$, with $\sigma_i$ being the classic generators of the braid group given by Artin \cite{Art}, for $1\leq i \leq m-1$. Given $k \in \mathbb{Z}$, denote by $T^{(k)}(m,n)$ the link represented by the closed braid $\widehat{\beta}_{m,n} = (\sigma_{m-1} \sigma_{m-2} \cdots \sigma_2 \sigma_1)^n\sigma_1^{k}$, that is, the link $T(m,n)$ after performing a $t_k$-move on the first two strands (Figure \ref{tkmove}). Unless otherwise stated, we will keep the notation $T^{(k)}(m,n)$ when referring to the classical diagram representing the link $T^{(k)}(m,n)$ (that is, the one expresed as the previous closed braid). See Figure \ref{exampletorus} for examples and braid conventions.

\begin{figure}[h!]
\centering
\includegraphics[width = 12.5cm]{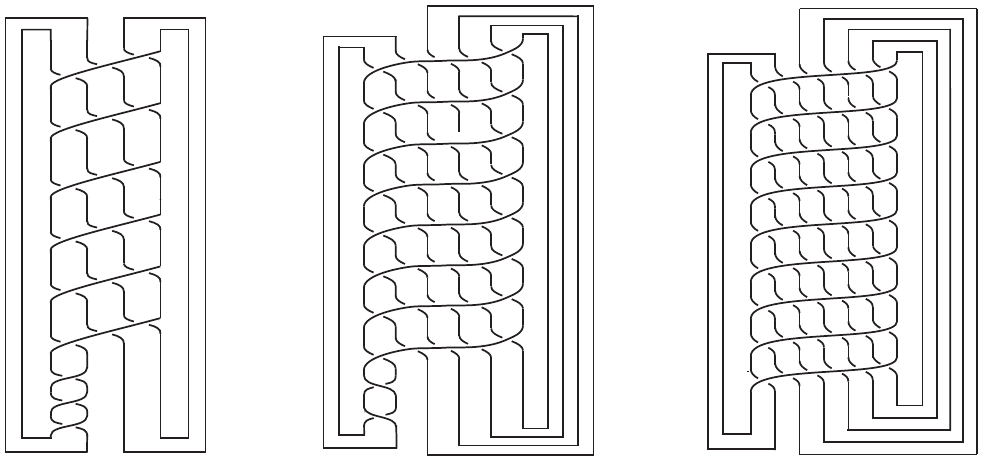}
\caption{\small{The torus links $T^{(3)}(4,6)$, $T^{(-2)}(6,8)$ and $T(7,9)$.}}
\label{exampletorus}
\end{figure}

\begin{theorem}\label{Theorem 4.1} Let $D$ be a link diagram with writhe $w(D) = w$ and $B$-smoothing number equaling one in such a way that $D_B$ is the trivial knot diagram with framing $w_B = w(D_B)$. Let $u = \frac{w-w_B+1}{2}$, then
\
\begin{enumerate}
  \item For any pair of integers $(i,j) \notin \{(u,3u),\, (u,3u-2), \, (u-1,3u), \, (u-1,3u-2)\}$, $$H^{i,j}(D) = H^{i,j-1}(D_A).$$
  \item \emph{(a)} $$
\begin{array}{ccccccccc} && 0 &  {\longrightarrow} & H^{u-1,3u}(D) &
\stackrel{\beta_*} {\longrightarrow} & H^{u-1,3u-1}(D_A) & \stackrel{\gamma_*} {\longrightarrow} \\
& \stackrel{\gamma_*} {\longrightarrow} & \mathbb{Z} & \stackrel{\alpha_*} {\longrightarrow} & H^{u,3u}(D) &
\stackrel{\beta_*} {\longrightarrow} & H^{u,3u-1}(D_A) & {\longrightarrow} &0.
\end{array}
$$
\emph{(b)} $$
\begin{array}{ccccccccc} && 0 &  {\longrightarrow} & H^{u-1,3u-2}(D) &
\stackrel{\beta_*} {\longrightarrow} & H^{u-1,3u-3}(D_A) & \stackrel{\gamma_*} {\longrightarrow} \\
& \stackrel{\gamma_*} {\longrightarrow} & \mathbb{Z} & \stackrel{\alpha_*} {\longrightarrow} & H^{u,3u-2}(D) &
\stackrel{\beta_*} {\longrightarrow} & H^{u,3u-3}(D_A) & {\longrightarrow} &0.
\end{array}
$$
\item  For any pair of integers $(i,j) \notin \{(u,3u), (u,3u-2)\}$,
$$tor H^{i,j}(D) = tor H^{i,j-1}(D_A).$$
\end{enumerate}
\end{theorem}

\begin{proof} Khovanov (co)homology of the trivial knot is equal to $\mathbb{Z}$ when $(i,j) = (0, \pm 1)$. Therefore, $$H_{a,b}(D_B)=\left\{
                             \begin{array}{ll}
                               \mathbb{Z} & \hbox{if $(a,b)=(w_B,3w_B \pm 2)$,} \\
                               0 & \hbox{otherwise.}
                             \end{array}
                           \right.$$\\
(1) Since $D_B$ is the trivial knot, $H^{i-u,j-3u+1}(D_B)=H^{i-u+1,j-3u+1}(D_B)=0$, and therefore the associated long exact sequence becomes of the form
$$0 \longrightarrow H^{i,j}(D) \longrightarrow H^{i,j-1}(D_A) \longrightarrow 0.$$ 
(2) The long exact sequences in (a) and (b) follow from combining the Khovanov homology of the trivial knot with Lemma \ref{Corollary 2.2.}. \ \\
\ \\
(3) In addition to (1), we need the equality of torsion for indices $(i,j)=(u-1,3u)$ and $(u-1,3u-2).$ For $(i,j)=(u-1,3u),$ the sequence
$$ 0  \longrightarrow  H^{u-1,3u}(D)
\stackrel{\beta_*} {\longrightarrow}  H^{u-1,3u-1}(D_A)  \stackrel{\gamma_*} {\longrightarrow} \gamma_* (H^{u-1,3u-1}(D_A))  \longrightarrow 0$$
splits because $\gamma_* (H^{u-1,3u-1}(D_A))$ is isomorphic to either $0$ or $\mathbb{Z}$. Therefore, $H^{u-1,3u-1}(D_A)=H^{u-1,3u}(D)$ or $H^{u-1,3u-1}(D_A)=H^{u-1,3u}(D) \oplus \mathbb{Z}$.
An analogous argument works for the case $(i,j)=(u-1,3u-2)$, where $H^{u-1,3u-3}(D_A)=H^{u-1,3u-2}(D)$ or $H^{u-1,3u-3}(D_A)=H^{u-1,3u-2}(D) \oplus \mathbb{Z}.$
\end{proof}
An example of a link diagram $D$ with $D_{B}$ the trivial knot is illustrated by the link $T^{(k)}(m,sm \pm 2)$ for any integer $k$, $m\geq 2$ and $s\geq0.$ See Figure \ref{Tk46toTrivial}.

\subsection{Khovanov homology of T\textsuperscript{(k)}(m,m+2)}\label{3.1}
We observe that the link $T^{(k)}_{B}(m,m+2)$ is equivalent to the trivial knot with writhe
$$w(m,k)=w(T^{(k)}_{B}(m,m+2))=-1-2\lfloor\frac{m}{2}\rfloor-k=\left\{
                                                                 \begin{array}{ll}
                                                                   -m-k, & \hbox{if $m$ is odd;} \\
                                                                   -1-m-k, & \hbox{if $m$ is even.}
                                                                 \end{array}
                                                               \right.
$$

\ \\
\begin{figure}[ht]
\centering
\includegraphics[scale=1]{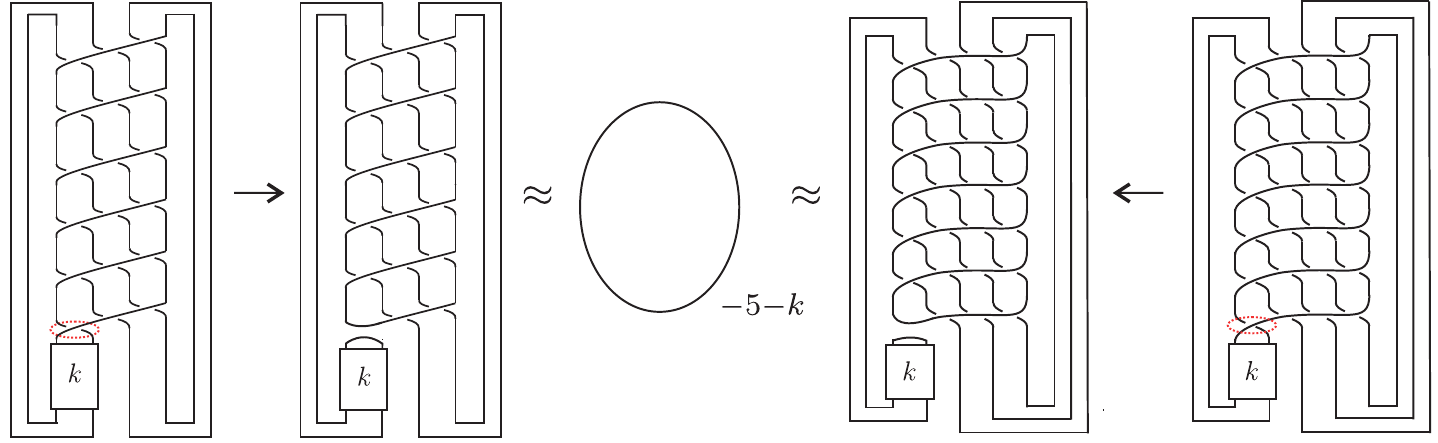}
\caption{The reductions of $T^{(k)}_{B}(m,m+2)$ to the framed trivial knot $\bigcirc_{-5-k}$ for the cases $m=4,5$.}
\label{Tk46toTrivial}
\end{figure}

The sub-index $\nu$ in $D_{\nu}$ means that the blackboard framing of $D$ was changed by $\nu$, in particular the writhe number satisfies $w(D_{\nu})=w(D)+\nu$ (compare Figure \ref{Tk46toTrivial}).

The following results follow from Theorem \ref{Theorem 4.1}.
\begin{corollary}\label{Khovanov T(m,m+2)}
Let $u = u(m,k) = \lfloor \frac{m(m+2)}{2} \rfloor +k$. Then
\begin{enumerate}
  \item For any $k \in \mathbb{Z}$ and index $(i,j) \notin \{(u,3u),\, (u,3u-2), \, (u-1,3u), \, (u-1,3u-2)\},$ 
$$H^{i,j}(T^{(k)}(m,m+2))=H^{i,j-1}(T^{(k-1)}(m,m+2)).$$ 
\item  For any pair of integers $(i,j) \notin \{(u,3u) ,(u,3u-2)\},$
$$tor H^{i,j}(T^{(k)}(m,m+2)) = tor H^{i,j-1}(T^{(k-1)}(m,m+2)).$$
\end{enumerate}
%\begin{enumerate}
%  \item if $m$ is even, we have
%  $$H^{i,j}(T^{(k)}(m,m+2))=H^{i,j-1}(T^{(k-1)}(m,m+2))$$
%  where $i \neq \frac{(m-1)(m+2)+m+2k+1\pm1}{2}$ and $j \neq \frac{3\{(m-1)(m+2)+m+2k+1\}+1\pm2}{2},$
%  \item if $m$ is odd, then
%      $$H^{i,j}(T^{(k)}(m,m+2))=H^{i,j-1}(T^{(k-1)}(m,m+2))$$
%      where $i \neq \frac{(m-1)(m+2)+m+2k\pm1}{2}$ and $j \neq \frac{3\{(m-1)(m+2)+m+2k\}+1\pm2}{2}.$
%\end{enumerate}

\end{corollary}

%\begin{proof}
%(1) Suppose that $m$ is even. Since the $A$-smoothing of a link diagram of $T^{(k)}(m,m+2)$ at the twist is equivalent to $T^{(k-1)}(m,m+2),$ we can obtain the %following long exact sequence of framed Khovanov homology:
%$$\cdots \rightarrow H_{a+1,b+1}(T^{(k)}_{B}(m,m+2)) \rightarrow H_{a,b}(T^{(k)}(m,m+2)) \rightarrow H_{a-1,b-1}(T^{(k-1)}(m,m+2))$$
%$$\rightarrow H_{a-1,b+1}(T^{(k)}_{B}(m,m+2)) \rightarrow  \cdots.$$
%Since the diagram of $T^{(k)}_{B}(m,m+2)$ is equivalent to the trivial knot diagram with writhe $-(m+1)-k,$ $H_{a,b}(T^{(k)}_{B}(m,m+2))=\mathbb{Z}$ if %$a=-(m+1)-k,~ b=-3(m+1)-3k \pm 2.$ Therefore, if $a\neq-(m+1)-k\pm1,~b\neq-3(m+1)-3k-1\pm2,$ we have
%$$H_{a,b}(T^{(k)}(m,m+2))=H_{a-1,b-1}(T^{(k-1)}(m,m+2)).$$
%
%(2) Assume that $m$ is odd. In a similar way to the proof (1), we can obtain the following long exact sequence of framed Khovanov homology:
%$$\cdots \rightarrow H_{a+1,b+1}(T^{(k)}_{B}(m,m+2)) \rightarrow H_{a,b}(T^{(k)}(m,m+2)) \rightarrow H_{a-1,b-1}(T^{(k-1)}(m,m+2))$$
%$$\rightarrow H_{a-1,b+1}(T^{(k)}_{B}(m,m+2)) \rightarrow \cdots.$$
%Since the diagram of $T^{(k)}_{B}(m,m+2)$ is equivalent to the trivial knot diagram with writhe $-m-k,$ $H_{a,b}(T^{(k)}_{B}(m,m+2))=\mathbb{Z}$ if $a=-m-k,~ %b=-3m-3k \pm 2.$ Therefore, if $a\neq-m-k\pm1,~b\neq-3m-1-3k\pm2,$ we have the following equation:
%$$H_{a,b}(T^{(k)}(m,m+2))=H_{a-1,b-1}(T^{(k-1)}(m,m+2)).$$

%\end{proof}

Corollary \ref{Khovanov T(m,m+2)} allows us to find infinite families of knots and 2-component links whose Khovanov homology groups contain torsion different from $\mathbb{Z}_{2}.$

\begin{corollary}\label{T^{(k)}(4,6)}
The Khovanov homology of $T^{(k)}(4,6)$ contains $\mathbb{Z}_{4}$-torsion for $k \geq -3$. Specifically, $H^{9,28+k}(T^{(k)}(4,6))$ has $\mathbb{Z}_{4}$-torsion for $k \geq -3$.
\end{corollary}

\begin{proof}
The Khovanov homology of the torus link $T(4,6)$ has $\mathbb{Z}_{4}$-torsion in bidegree $(9,28)$, as shown in Table \ref{T64}. By Corollary \ref{Khovanov T(m,m+2)} (2) one checks that $\mathbb{Z}_{4}$ survives in bidegree $(9,28+k)$ if $k \geq -3.$ When $k=-3$, the pair $(i,j)=(9,25)$ is a critical pair. In fact, $H^{9,24}(T^{(-4)}(4,6))$ has no $\mathbb{Z}_{4}$-torsion as shown in Table \ref{T64m4}).
\end{proof}

\begin{corollary}\label{T^{(k)}(5,7)}

The Khovanov homology of $T^{(k)}(5,7)$ contains $\mathbb{Z}_{5}$-torsion for every integer $k.$ \\Specifically, $H^{11,39+k}(T^{(k)}(5,7))$ and $H^{12,43+k}(T^{(k)}(5,7))$  have $\mathbb{Z}_{5}$-torsion for all $k$.

\end{corollary}

\begin{proof}

The Khovanov homology of the torus knot $T(5,7)$ has $\mathbb{Z}_{5}$-torsion in bidegree $(11,39)$ and $(12,43)$ (see Table \ref{T75}). Using Corollary \ref{Khovanov T(m,m+2)} (2), we check that if $k \geq -6,$ then $H^{11,39+k}(T^{(k)}(5,7))$ contains $\mathbb{Z}_{5}$-torsion.\footnote{We say that the finitely generated abelian group $G$ has $\mathbb{Z}_{p^{i}}$-torsion if $\mathbb{Z}_{p^{i}}$ is a summand in its primary decomposition.} From computation (Table \ref{T75m7}),  $H^{11,32}(T^{(-7)}(5,7))$ has $\mathbb{Z}_{5}$-torsion, and by using Corollary \ref{Khovanov T(m,m+2)} (2) again, we have that if $k \leq -8,$ then $H^{11,39+k}(T^{(k)}(5,7))$ contains $\mathbb{Z}_{5}$-torsion. In summary, we have proved that $H^{11,39+k}(T^{(k)}(5,7))$ contains $\mathbb{Z}_{5}$-torsion for every $k.$\ \\
$H^{12,43+k}(T^{(k)}(5,7))$ has $\mathbb{Z}_{5}$-torsion for every integer $k$ follows directly from Corollary \ref{Khovanov T(m,m+2)} (2).
\end{proof}

\begin{corollary}\label{T^{(k)}(6,8)}
The Khovanov homology of $T^{(k)}(6,8)$ contains
\begin{enumerate}
\item{$\mathbb{Z}_{3},\mathbb{Z}_{4},\mathbb{Z}_{5}$-torsion for every integer $k$. In particular, $\mathbb{Z}_3$-torsion appears in \\ $H^{19,60+k}(T^{(k)}(6,8))$ and $H^{20,64+k}(T^{(k)}(6,8))$ for every integer $k.$}
\item{$\mathbb{Z}_{8}$-torsion for $k \leq -8$}. Specifically, $H^{17,58+k}(T^{(k)}(6,8))$ has $\mathbb{Z}_{8}$-torsion for $k\leq -8$.
\end{enumerate} 
\end{corollary}

\begin{proof}\ \\
(1) The Khovanov homology of the torus link $T(6,8)$ contains $\mathbb{Z}_{3}$-torsion in bidegrees $(19,60)$ and $(20,64)$ (see Table \ref{T86}). We can show that if $k \geq -5,$ then the Khovanov homology of $T^{(k)}(6,8)$ has $\mathbb{Z}_{3}$-torsion in bidegree $(19,60+k)$ by Corollary \ref{Khovanov T(m,m+2)}(2). When $k =-6,$ $H^{19,54}(T^{(-6)}(6,8))$ (see Table \ref{T86m6}) has $\mathbb{Z}_3$-torsion. Using Corollary \ref{Khovanov T(m,m+2)}(2), we prove that \\ $H^{19,60+k}(T^{(k)}(6,8))$ contains $\mathbb{Z}_3$-torsion for $k \leq -7.$ In a similar way, we can prove that the Khovanov homology of $T^{(k)}(6,8)$ has $\mathbb{Z}_3$-torsion in bidegree $(20,64+k)$ for all $k$. See Table \ref{T86m5}.\

The Khovanov homology of the torus link $T(6,8)$ contains $\mathbb{Z}_{4}$-torsion in bidegree $(18,62)$ (see Table \ref{T86}). The Khovanov homology of $T^{(k)}(6,8)$ has $\mathbb{Z}_{4}$-torsion in bidegree $(18,62+k)$ by Corollary \ref{Khovanov T(m,m+2)}(2).\

The Khovanov homology of the torus link $T(6,8)$ contains $\mathbb{Z}_{5}$-torsion in bidegree $(12,54)$ (see Table \ref{T86}). The Khovanov homology of $T^{(k)}(6,8)$ has $\mathbb{Z}_{5}$-torsion in bidegree $(12,54+k)$ by Corollary \ref{Khovanov T(m,m+2)}(2).\ \\
\ \\
$(2)$ In the case of $\mathbb{Z}_{8}$-torsion, we compute that $H^{17,50}(T^{(-8)}(6,8))=\mathbb{Z}_{2}^{2}\oplus \mathbb{Z}_{3} \oplus \mathbb{Z}_{8}$ (see Table \ref{T86m8}). By Corollary \ref{Khovanov T(m,m+2)}(1), the Khovanov homology of $T^{(k)}(6,8)$ has this group in bidegree $(17,58+k)$ for $k\leq -8$.
\end{proof}

\begin{corollary}\label{T^{(k)}(7,9)} \ \\
The Khovanov homology of $T^{(k)}(7,9)$ contains $\mathbb{Z}_{3},\mathbb{Z}_{4},\mathbb{Z}_{5},\mathbb{Z}_{7}$-torsion for every integer $k$. In particular, $H^{15,67+k}(T^{(k)}(7,9))$, $H^{21,75+k}(T^{(k)}(7,9))$, $H^{22,79+k}(T^{(k)}(7,9))$, $H^{23,77+k}(T^{(k)}(7,9))$, and $H^{24,81+k}(T^{(k)}(7,9))$ have $\mathbb{Z}_{7}$-torsion for all $k$.
\end{corollary}

\begin{proof}

It is known that the Khovanov homology of the torus knot $T(7,9)$ has $\mathbb{Z}_{3}$-torsion in bidegrees $(26,85),(28,87),(29,91)$ (see Table \ref{T97}). The Khovanov homology of $T^{k}(7,9)$ has $\mathbb{Z}_{3}$-torsion in bidegrees $(26,85+k),(29,91+k)$ following directly from Corollary \ref{Khovanov T(m,m+2)}(2), and in bidegree $(28,87+k)$ following a similar proof as Corollary \ref{T^{(k)}(6,8)}(1). See Table \ref{T97m4}.\ \\

The Khovanov homology of the torus knot $T(7,9)$ has $\mathbb{Z}_{4}$-torsion in bidegrees $(21,75),\ (22,79),$ and $(26,85)$  (see Table \ref{T97}). The Khovanov homology of $T^{k}(7,9)$ has $\mathbb{Z}_{4}$-torsion in bidegrees $(21,75+k),(26,85+k)$, and $(22,79+k)$ following from Corollary \ref{Khovanov T(m,m+2)}(2).\ \\

The Khovanov homology of the torus knot $T(7,9)$ has $\mathbb{Z}_{5}$-torsion in bidegrees $(12,67),$ $(22,79)$, $(24,81)$, and $(25,85)$  (see Table \ref{T97}). The Khovanov homology of $T^{k}(7,9)$ has $\mathbb{Z}_{5}$-torsion in bidegrees $(12,67+k),(22,79+k),(25,85+k)$, and $(24,81+k)$ following from Corollary \ref{Khovanov T(m,m+2)}(2).\ \\

The Khovanov homology of the torus knot $T(7,9)$ has $\mathbb{Z}_{7}$-torsion in bidegrees $(15,67),\ (21,75),\\ (22$ ,$79)$,
$(23,77),$ and $(24,81)$  (see Table \ref{T97}). The Khovanov homology of $T^{k}(7,9)$ has $\mathbb{Z}_{7}$-torsion in bidegrees $(15,67+k),(22,79+k),$ $(21,75+k),$ and $(24,81+k)$ following from Corollary  \ref{Khovanov T(m,m+2)}(2). In a similar way as the proof of Corollary \ref{T^{(k)}(6,8)}(1), one can show that the Khovanov homology of $T^{(k)}(7,9)$ has $\mathbb{Z}_{7}$-torsion in bidegree $(23,77+k)$ for every $k$. See Table \ref{T97m9}.
\end{proof}

We conjecture about the shape of the Khovanov homology of $H^{i,j}(T^{(k)}(m,m+2))$ for $k \geq 0.$ We start from the case of $k=0.$

\begin{conjecture}\label{Conjecture 3.5}
\
\begin{enumerate}
\item[(1)]
$H^{i,j}(T(m,m+2))=0$ if $i>u(m,0)$ or $j>3u(m,0).$
\item[(2)]
\begin{equation*}\label{1.1}
H^{u(m,0),3u(m,0)}(T(m,m+2))=
\begin{cases*}
 \mathbb{Z} & if m is odd or is equal to 2,\\
 \mathbb{Z}^2& otherwise.
\end{cases*}
\end{equation*}
\item[(3)]
For $m\geq 4$, the Khovanov homology of the torus link $T(m,m+2)$ has $\mathbb{Z}_{m}$-torsion in some bidegree with $i<u(m,0)$ and $j<3u(m,0)$.
\end{enumerate}
\end{conjecture}

From computational data, Conjecture \ref{Conjecture 3.5} holds when $2\leq m<8$.

\begin{corollary} \label{y}
Assuming Conjecture \ref{Conjecture 3.5} holds for fixed $m$, we have
\begin{enumerate}
  \item $H^{i,j}(T^{(k)}(m,m+2))=0$ if $i>u(m,k)$ or $j>3u(m,k)$.
  \item For $k>0,$ $H^{u(m,k),3u(m,k)}(T^{(k)}(m,m+2))=\mathbb{Z}$.
  \item For $k\geq0$ and $m \geq 4,$ the Khovanov homology of the torus link $T^{(k)}(m,m+2)$ has $\mathbb{Z}_{m}$-torsion.
\end{enumerate}
\end{corollary}

\begin{proof}
Using Conjecture \ref{Conjecture 3.5} and the long exact sequence of Khovanov homology, the proof proceeds in a similar way as the proof of Theorem \ref{Theorem 4.1}.
\end{proof}

We continue with a conjecture on how to obtain the entire Khovanov homology of\ \\ $T^{(k)}(m,m+2)$ from the Khovanov homology of $T^{(k-1)}(m,m+2).$

\begin{conjecture}\label{Conjecture 3.7} For $k > 0,$ $H^{i,j}(T^{(k)}(m,m+2))=H^{i,j-1}(T^{(k-1)}(m,m+2))$ with the following exceptions:\\
(1) $H^{u(m,k), 3u(m,k)}(T^{(k)}(m,m+2)) = \mathbb{Z}$ while $H^{u(m,k), 3u(m,k)}(T^{(k-1)}(m,m+2))=0.$ \\
(2) We have two cases depending on the parity of $(m-k)$:
\begin{enumerate}
\item[(a)]
for even $(m-k),$ we have $H^{u(m,k), 3u(m,k)-2}(T^{(k)}(m,m+2)) = \mathbb{Z}$\\ while $H^{u(m,k), 3u(m,k)-3}(T^{(k-1)}(m,m+2)) =0,$
\item[(b)]
for odd $(m-k),$ we have $H^{u(m,k), 3u(m,k)-2}(T^{(k)}(m,m+2)) = \mathbb{Z}_2$\\ while $H^{u(m,k), 3u(m,k)-2}(T^{(k-1)}(m,m+2)) =0$ and \\
$H^{u(m,k)-1,3u(m,k)-2}(T^{(k)}(m,m+2))\oplus \mathbb{Z} = H^{u(m,k)-1,3u(m,k)-3}(T^{(k-1)}(m,m+2)).$
\end{enumerate}
\end{conjecture}
The conjecture holds for $m=2$ \cite{Kho1}. We have verified this conjecture for $m=3,4,5,6,7,$ and `small' $k$.
We can partially prove the conjecture analyzing the following exceptional indices using Conjecture \ref{Conjecture 3.5} and Corollary \ref{y}:
$$(i,j)= (u,3u),(u,3u-2),(u-1,3u), (u-1,3u-2) \mbox{ where } u=u(m,k)=\lfloor\frac{m(m+2)}{2}\rfloor +k.$$
To analyze Khovanov homology at these indices, we consider the following pieces of the long exact sequence of Khovanov homology\footnote{We write $T^{(k)}$ for $T^{(k)}(m,n)$ in the following long exact sequences.}:
\\(1)
$$0 \longrightarrow H^{u-1,3u}(T^{(k)}) \stackrel{\beta_*}{\longrightarrow} H^{u-1,3u-1}(T^{(k-1)}) \stackrel{\gamma_*}{\longrightarrow} \mathbb{Z}
\stackrel{\alpha_*}{\longrightarrow} H^{u,3u}(T^{(k)})\stackrel{\beta_*}{\longrightarrow} H^{u,3u-1}(T^{(k-1)}) \longrightarrow 0,$$ and
\\(2)
$$0 \to H^{u-1,3u-2}(T^{(k)}) \stackrel{\beta_*}{\longrightarrow} H^{u-1,3u-3}(T^{(k-1)}) \stackrel{\gamma_*}{\longrightarrow} \mathbb{Z} \stackrel{\alpha_*}{\longrightarrow} H^{u,3u-2}(T^{(k)}) \stackrel{\beta_*}{\longrightarrow} H^{u,3u-3}(T^{(k-1)}) \to 0.$$

By Corollary \ref{y}, we have $H^{u-1,3u-1}(T^{(k-1)})=H^{u,3u-1}(T^{(k-1)})=0.$ Thus, sequence $(1)$ reduces to:
$$0 \longrightarrow H^{u-1,3u}(T^{(k)}) \longrightarrow 0 \longrightarrow \mathbb{Z} \longrightarrow H^{u,3u}(T^{(k)}) \longrightarrow 0.$$
Thus, $H^{u-1,3u}(T^{(k)})=0$ and $H^{u,3u}(T^{(k)})=\mathbb{Z}$. \\

Again, by Corollary \ref{y} assuming $k>1$, or $k=1$ and $m$ odd, we have $H^{u-1,3u-3}(T^{(k-1)})=\mathbb{Z}$ and $H^{u,3u-3}(T^{(k-1)})=0.$ Thus sequence $(2)$ transforms into, $$0 \longrightarrow H^{u-1,3u-2}(T^{(k)}) \stackrel{\beta_*}{\longrightarrow} \mathbb{Z} \stackrel{\gamma_*}{\longrightarrow} \mathbb{Z} \stackrel{\alpha_*}{\longrightarrow} H^{u,3u-2}(T^{(k)})\longrightarrow 0.$$
This sequence gives various choices for $H^{u,3u-2}(T^{(k)})$ and $H^{u-1,3u-2}(T^{(k)})$ depending on the function $\gamma_*:\mathbb{Z} \longrightarrow \mathbb{Z}$. Let $\gamma _*(1)=s$, $s\geq 0$. Then, we have two basic possibilities:\\
(i) $s=0$. Then, $H^{u,3u-2}(T^{(k)})=\mathbb{Z}=H^{u-1,3u-2}(T^{(k)})$. Until now, we assumed that Conjecture \ref{Conjecture 3.5} holds. Conjecture \ref{Conjecture 3.7} implies that this case holds when $(m-k)$ is even.\\
(ii) $s>0$. Then, $\gamma_*$ is a monomorphism. Hence, $H^{u-1,3u-2}(T^{(k)})=0$ and $H^{u,3u-2}(T^{(k)})=\mathbb{Z}_s$. Until now, we assumed that Conjecture \ref{Conjecture 3.5} holds. Conjecture \ref{Conjecture 3.7} implies that this case holds when $(m-k)$ is odd with $s=2$.

\subsection{Khovanov homology of T\textsuperscript{(k)}(m,m+1) for m = 4, 5, 7 }\label{3.2}

The torus knots $T(m,m+1)$ were the first knots in which torsion different from ${\mathbb Z}_2$ was
observed, namely ${\mathbb Z}_4$ in $H^{**}(T(4,5))$, ${\mathbb Z}_3$ and ${\mathbb Z}_5$ in  $H^{**}(T(5,6))$,
and ${\mathbb Z}_7$ in $H^{**}(T(7,8))$.  However, finding  Khovanov
homology of $T^{(k)}(m,m+1)$ is more difficult than that for  $T^{(k)}(m,m+2)$
as $T^{(k)}_{B}(m,m+1)$ is not the trivial knot.
In this subsection, we will prove existence of ${\mathbb Z}_4$-torsion in the Khovanov homology of $T^{(k)}(4,5)$ when $k \geq 0$
and ${\mathbb Z}_5$-torsion in the Khovanov homology of $T^{(k)}(5,6)$ for  $k \geq 0$.

Before this, we observe the coincidences between $T(m,m+1)$ and $T(m,m+2)$ for properly
chosen twists.

\begin{proposition} \label{prop}
        \item[(1)] $T^{(k)}(3,4)= T^{(k-2)}(3,5)$. These links have $\mathbb{Z}_3$-torsion
         in odd Khovanov homology for $k \geq 0$ (see Section \ref{Section 5}).
	\item[(2)] $T^{(2)}(4,5)= T^{(-1)}(4,6)$, with ${\mathbb Z}_4$-torsion in Khovanov homology.
	\item[(3)] $T^{(2)}(5,6)= T^{(-2)}(5,7)$, with ${\mathbb Z}_5$-torsion in Khovanov homology (see Figure \ref{Jozef-reduction}
	illustrating the equivalence of these knots).
	\end{proposition}

\begin{figure}[h!]
	\centering
	\includegraphics[scale=1.4]{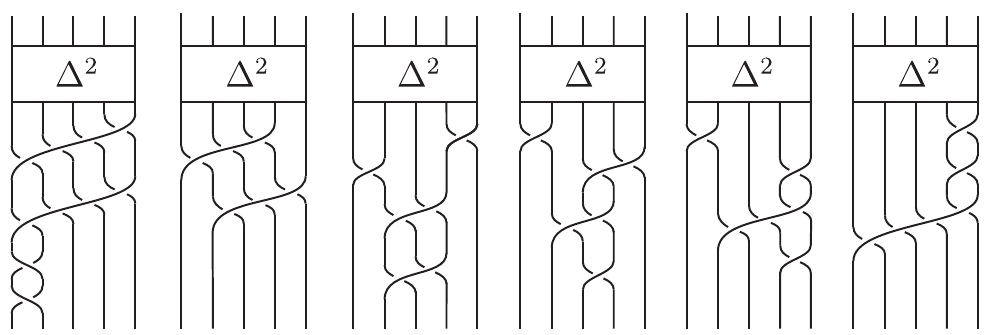}
	\caption{\small{Operations on braids allowing conjugacy leading from $T^{(-2)}(5,7)$ to $T^{(2)}(5,6)$.}}
	\label{Jozef-reduction}
\end{figure}

\begin{theorem}\label{Theorem 3.11}\
\begin{enumerate}
\item[(1)] The Khovanov homology of $T^{(k)}(4,5)$ contains ${\mathbb Z}_4$-torsion for any $k\geq 0$.  We have
\begin{equation*}
tor H^{9,25+k}(T^{(k)}(4,5))=
\newline
\begin{cases}
{\mathbb Z}_4 & \text{if $k=0$, or $1$}\\
{\mathbb Z}_2 \oplus {\mathbb Z}_4& \text{if $k\geq 2$,}\\
0 & \text{otherwise.}
\end{cases}
\end{equation*}
\item[(2)] The Khovanov homology of $T^{(k)}(5,6)$ contains ${\mathbb Z}_5$-torsion for any $k \geq 0$. Specifically,\\ $H^{11,35+k}(T^{(k)}(5,6))$ and $H^{12,39+k}(T^{(k)}(5,6))$
have ${\mathbb Z}_5$-torsion for any $k \geq 0$.
\item[(3)] The Khovanov homology of $T^{(k)}(7,8)$ contains ${\mathbb Z}_7$-torsion for any $k \neq -1$. Specifically, \\ $H^{15,61+k}(T^{(k)}(7,8))$ contains $\mathbb {Z}_7$-torsion for any $k \geq 0$, and $H^{23,73+k}(T^{(k)}(7,8))$ contains $\mathbb {Z}_7$-torsion for any $k \leq -2$.\footnote{Notice that the Khovanov homology of $T^{(-1)}(7,8)$ does not contain $\mathbb{Z}_7$-torsion, see Table \ref{T87m1}.}
\end{enumerate}
\end{theorem}

\begin{proof}
We concentrate on the proofs of cases (1) and (2).
Our starting point is the long exact sequence of Lemma \ref{Corollary 2.2.}.
In this case, we have $D=T^{(k)}=T^{(k)}(4,5)$, $D_A=T^{(k-1)}=T^{(k-1)}(4,5)$, $D_B=T(2,3)_{-1-k}$, and $w(D)=15+k$,  $w(D_A)=14+k$, $w(D_B)=2-k$ .

Thus the long exact sequence above now has the form

$$
\begin{array}{ccccccc}
\cdots& \stackrel{\gamma_*} {\longrightarrow} & H^{-8-k+i,-20-3k+j}(T(2,3)) & \stackrel{\alpha_*} {\longrightarrow} & H^{i-1,j}(T^{(k)}) &
\stackrel{\beta_*} {\longrightarrow} & H^{i-1,j-1}(T^{(k-1)}) \\
& \stackrel{\gamma_*} {\longrightarrow} & H^{-7-k+i,-20-3k+j}(T(2,3)) & \stackrel{\alpha_*} {\longrightarrow} & H^{i,j}(T^{(k)}) &
\stackrel{\beta_*} {\longrightarrow} & H^{i,j-1}(T^{(k-1)}) \\
& \stackrel{\gamma_*} {\longrightarrow} & H^{-6-k+i,-20-3k+j}(T(2,3)) & \stackrel{\alpha_*} {\longrightarrow} &\cdots.
\end{array}
$$ 
The homology of the right handed trefoil knot $T(2,3)$ is given by
\begin{equation*}
H^{i',j'}(T(2,3))=
\newline
\begin{cases}
{\mathbb Z} & \text{if $(i',j')\in \{(0,1),(0,3),(2,5),(3,9)\}$,}\\
{\mathbb Z}_2 & \text{if $(i',j')=(3,7)$,}\\
0 & \text{otherwise.}
\end{cases}
\end{equation*}
From this we see that $tor H^{i,j}(T^{(k)}(4,5))= tor H^{i,j-1}(T^{(k-1)}(4,5))$ with possible exceptions when $$H^{-7-k+i,-20-3k+j}(T(2,3))\neq 0 \mbox{ or } H^{-6-k+i,-20-3k+j}(T(2,3))={\mathbb Z}_2,$$ That is when $(i,j)\in \{(7+k,21+3k),(7+k,23+3k),(9+k,25+3k),(10+k,29+3k),(10+k,27+3k),(9+k,27+3k)\}$.
From Tables \ref{T54} and \ref{T542}, we know that $H^{9,25}(T(4,5)) = {\mathbb Z}_4$ and $H^{9,27}(T^{(2)}(4,5))={\mathbb Z}_2\oplus {\mathbb Z}_4$,
and the only critical value concerning us is $(i,j)= (9,27)$ for $k=2$. Thus $tor H^{9,26}(T^{(1)}(4,5))= H^{9,25}(T(4,5)) = {\mathbb Z}_4$, and
$H^{9,25+k}(T^{(2)}(4,5))=H^{9,27}(T^{(2)}(4,5))={\mathbb Z}_2\oplus {\mathbb Z}_4$, for $k\geq 2$, as needed.\\

In a similar manner we prove that $H^{11,35+k}(T^{(k)}(5,6))$ and $H^{12,39+k}(T^{(k)}(5,6))$ have $\mathbb{Z}_5$-torsion for $k\geq 0$. We will show, analyzing homology of $H^{**}(T(3,4))$ that $$tor_{odd}H^{11,35+k}(T^{(k)}(5,6))= tor_{odd}H^{12,39+k}(T^{(k)}(5,6))=\mathbb{Z}_5,$$ where $tor_{odd}(G)$ denotes the odd part of torsion of $G$.\\
Figure \ref{T56toT34} illustrates $T^{(k)}_B(5,6)= T(3,4)_{-1-k}.$ 
\begin{figure}[h!]
\centering
\includegraphics[scale=1.2]{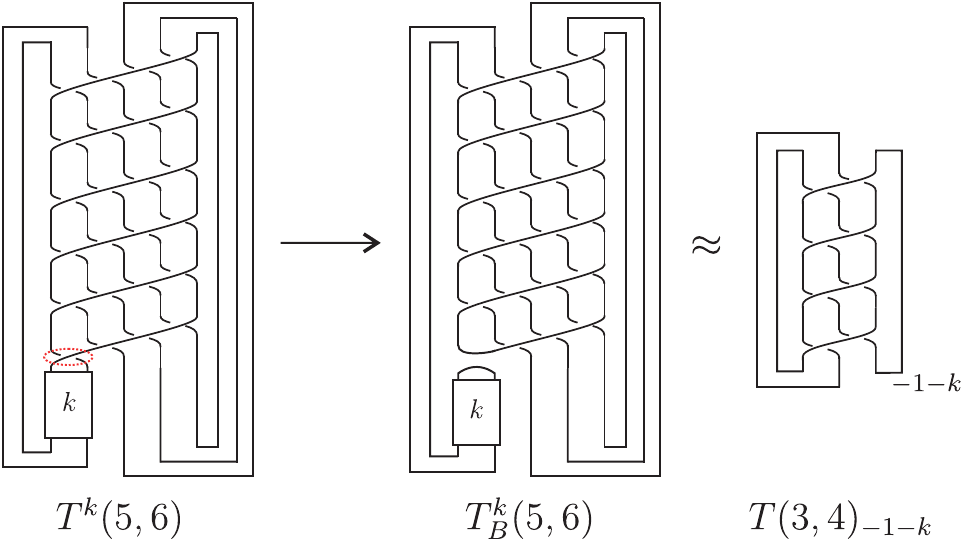} 
\caption{\small{The link $T^{(k)}(5,6)$ becomes $T(3,4)$ with framing $(-1-k)$ after performing a $B$-smoothing at the encircled crossing.}}
\label{T56toT34}
\end{figure}

Futhermore we have:
\begin{equation*}
H^{i',j'}(T(3,4))=
\newline
\begin{cases}
{\mathbb Z} & \text{if $(i',j')\in \{(0,5),(0,7),(2,9),(3,13),(4,11),(4,13),(5,15),(5,17) \}$,}\\
{\mathbb Z}_2 & \text{if $(i',j')=(3,11)$,}\\
0 & \text{otherwise.}
\end{cases}
\end{equation*}
Now we analyze the long exact sequence of homology of Lemma \ref{Corollary 2.2.} with
 $D=T^{(k)}=T^{(k)}(5,6)$, $D_A=T^{(k-1)}=T^{(k-1)}(5,6)$, $D_B=T(3,4)_{-1-k}$, and $w(D)=24+k$ $w(D_A)=23+k$, $w(D_B)=7-k$ .

Thus the interesting part of the long exact sequence now has the form
$$
\begin{array}{ccccccc}
\cdots & \stackrel{\gamma_*} {\longrightarrow} & H^{-9-k+i,-26-3k+j}(T(3,4)) & \stackrel{\alpha_*} {\longrightarrow} & H^{i,j}(T^{(k)}) &
\stackrel{\beta_*} {\longrightarrow} & H^{i,j-1}(T^{(k-1)}) \\
& \stackrel{\gamma_*} {\longrightarrow} & H^{-8-k+i,-26-3k+j}(T(3,4)) & \cdots.
\end{array}
$$
For the odd part of torsion in Khovanov homology the critical values, when possibly $tor_{odd}H^{i,j}(T^{(k)}) \\ \neq tor_{odd}H^{i,j-1}(T^{(k-1)})$ are:
$$(9+k,31+3k),(9+k,33+3k), (11+k,35+3k), (12+k,39+3k), (13+k,37+3k),$$
$$ (13+k,39+3k), (14+k,41+3k), (14+k,43+3k).$$
Notice that for $k=0$, $H^{11,35}(T(5,6)) = H^{12,39}(T(5,6)) = \mathbb{Z}_2 \oplus \mathbb{Z}_5,$ and $H^{14,43}(T(5,6)) = \mathbb{Z}_3$ (see Table \ref{T65}). Upper indices $(11,35),(12,39),$ and $(14,43)$ are critical values. In fact, $H^{**}(T^{(-1)}(5,6))$ has no odd torsion\footnote{We checked that $H^{*,*}(T^{(k)}(5,6))$ has no $\mathbb{Z}_3$-torsion for $-10 \leq k \leq -1$ and $2 \leq k \leq 10$, while $H^{14,43}(T(5,6)) = H^{14,44}(T^{(1)}(5,6)) = \mathbb{Z}_3.$ In fact, $(14,43)$ and $(14,45)$ are critical points for $k=0$ and $k=2$ respectively. See Tables \ref{T65}, \ref{T65m1}, \ref{T651} and \ref{T652}.}. See Table \ref{T65m1}.\\

For $k=2$ we have  critical value $(11,37)$ but from Table \ref{T652} we see that $tor_{odd}H^{11,37}(T^{(2)}(5,6))= {\mathbb Z}_5$ (see Table \ref{T652}). For $k>2$ there
is no critical value potentially changing ${\mathbb Z}_5$-torsion in bidegree $(11,35+k)$. When $k=3$, we have critical value $(12,42)$ but from Table \ref{T653} we see that  $tor_{odd}H^{12,42}(T^{(3)}(5,6)) = {\mathbb Z}_5.$ Thus, Theorem \ref{Theorem 3.11}(2) holds.\ \\
Case (3) follows similarly.  The tables of $T^{(k)}(7,8)$ for $k=-2,0,2$ (see Tables \ref{T87},\ref{T87m2}, and \ref{T872}) and table of $T(5,6)$ (Table \ref{T65}) are needed for the proof.
\end{proof}

\subsection{Small links with $\mathbb{Z}_3,\mathbb{Z}_ 4, \mathbb{Z}_5, \mathbb{Z}_7,$ and $\mathbb{Z}_8$-torsion}\label{ss 3.3}

We conclude the section summarizing the implications of the previous results.

$T(5,6)$ with crossing number $24$ was the smallest known knot with $\mathbb{Z}_3$-torsion. Computations show that the closure of the braid $(\sigma_3\sigma_2\sigma_1\sigma_4\sigma_3)^4$, a link of three components having at most $20$ crossings also has  $\mathbb{Z}_3$-torsion. See Tables \ref{T65} and \ref{32143_to_the power_4}. Moreover, the knot obtained by closing the braid $(\sigma_3\sigma_2\sigma_1\sigma_4\sigma_3)^4\sigma_4^{-1}\sigma_2^{-1}$ has at most $22$ crossings and has $\mathbb{Z}_3$-torsion. See Table \ref{32143_to_the power_4_-3-1}. The closure of the braids $(\sigma_3\sigma_2\sigma_1\sigma_1\sigma_2\sigma_3)^5$ and $(\sigma_3\sigma_2\sigma_1\sigma_1\sigma_2\sigma_3)^5\sigma_3\sigma_2$ with $4$ and $2$ components respectively have $\mathbb{Z}_3$-torsion in Khovanov homology. They are counterexamples to part $2$ of Conjecture \ref{PS} in the next subsection. See Tables \ref{4b4} and \ref{4b2}.

$T(4,5)$ with crossing number $15$ was the smallest prime knot with $\mathbb{Z}_4$-torsion \cite{Shu3}. From Corollary \ref{T^{(k)}(4,6)}, it follows that $T^{(-3)}(4,6)$, which is in fact the 2-cabling of trefoil with $3$-framing has $\mathbb{Z}_{4}$-torsion. Observe that, $T^{(-3)}(4,6)$ also has $15$ crossings.
Moreover, the closure of the braid $(\sigma_3\sigma_2\sigma_1)^7\sigma_1^{-5}\sigma_3^{-2},$ obtained from $T^{(-5)}(4,7)$ by adding two negative half twists on the right-most pair of strands is a two component link. After reduction it is the closure of the braid $\sigma_{2}\sigma_{1}^{2}(\sigma_{3}\sigma_{2})^{2}\sigma_{1}\sigma_{3}\sigma_{2}^{2}\sigma_{1}\sigma_{3}\sigma_{2}$ and has $\mathbb{Z}_4$-torsion. In particular, this link of two components has at most $14$ crossings, whereas the knots $T(4,5)$ and $T^{(-3)}(4,6)$ have $15$ crossings. See Tables \ref{T64} and \ref{T74m5_with-1-1}.

$T(5,6)$ with crossing number $24$ was the smallest known knot with $\mathbb{Z}_5$-torsion in Khovanov homology \cite{BN}. From Corollary \ref{T^{(k)}(5,7)}, it follows that the knot $T^{(-6)}(5,7)$, which is the closure of the braid $(\sigma_{3}\sigma_{2}\sigma_{1}\sigma_{4}\sigma_{3}\sigma_{2})^{3}\sigma_{4}\sigma_{3}\sigma_{2}$ of $22$ crossings has $\mathbb{Z}_{5}$-torsion.  It also follows that the link $T^{(-7)}(5,7)$, which is the closure of the braid $(\sigma_{3}\sigma_{2}\sigma_{1}\sigma_{4}\sigma_{3}\sigma_{2})^{3}\sigma_{4}\sigma_{3}$ of $21$ crossings has $\mathbb{Z}_{5}$-torsion. See Tables \ref{T75} and \ref{T75m7}. \iffalse Further, the knot obtained by closing the braid $(\sigma_3\sigma_2\sigma_1\sigma_4\sigma_3)^4\sigma_3$ with at most $21$ crossings has $\mathbb{Z}_5$-torsion. See Table \ref{32143_to_the power_4_3}.\fi

$T(7,8)$ with crossing number $48$ was the smallest known knot with $\mathbb{Z}_7$-torsion\footnote{Bar-Natan first computed this example \cite{PrSa}.}. From Corollary \ref{T^{(k)}(7,9)}, it follows that the knot $T^{(-8)}(7,9),$ which is the closure of the braid \\ $(\sigma_5\sigma_4\sigma_3\sigma_2\sigma_1\sigma_6\sigma_5\sigma_4\sigma_3\sigma_2)^4\sigma_6\sigma_5\sigma_4\sigma_3\sigma_2\sigma_1$  and the link $T^{(-9)}(7,9),$ which is the closure of the braid $(\sigma_5\sigma_4\sigma_3\sigma_2\sigma_1\sigma_6\sigma_5\sigma_4\sigma_3\sigma_2)^4\sigma_6\sigma_5\sigma_4\sigma_3\sigma_2$ have $\mathbb{Z}_{7}$-torsion. The knot $T^{(-8)}(7,9)$ and the link $T^{(-9)}(7,9)$ after reduction have crossing numbers at most $46$ and $45$ respectively. See Tables \ref{T97} and \ref{T97m9}.

The Khovanov homology of $T^{(-8)}(6,8)$ with braid index $6$ has $\mathbb{Z}_8$-torsion.\footnote{The only knot known previously with $\mathbb{Z}_{8}$-torsion was $T(8,9)$ \cite{Lew}.} After reduction, this link is equivalently the closure of the braid word $(\sigma_4\sigma_3\sigma_2\sigma_1\sigma_5\sigma_4\sigma_3\sigma_2)^4$ with $32$ crossings, and the 2-cabling of $T(3,4)$. Note that by Corollary \ref{T^{(k)}(6,8)} $T^{(-9)}(6,8)$, a knot of at most $33$ crossings, also has $\mathbb{Z}_8$-torsion. This is an example of new torsion being generated as a consequence of adding half twists to a link. Interestingly, $T(3,4)$ is also $8_{19}$ in the Rolfsen's knot table, the smallest non-alternating knot. See Figure \ref{cabling}. It is in fact a counterexample to part $3'$ of Conjecture \ref{PS} of the next subsection with $p=2$, $r=3$, and $n=6$. See Tables \ref{T86} and \ref{T86m8}.

\subsection{$\mathbb{Z}_{2^s}$-torsion in Khovanov homology}
Till now, no knot or link with torsion larger than $\mathbb{Z}_8$ was known. In this subsection we introduce infinite families of links with braid index $4$ containing $\mathbb{Z}_{2^s}$-torsion with $s\leq23.$ These infinite families also provide us with counterexamples to parts (2') and (3') of the following conjecture with $p = 2$. 

\begin{conjecture}[\cite{PrSa}]\label{PS}
	PS braid conjecture \
	\begin{enumerate}
		\item[(1)] Khovanov homology of a closed $3$-braid can have only $\mathbb{Z}_2$ torsion.
		\item[(2)] Khovanov homology of a closed $4$-braid cannot have an odd torsion.
		\item[(2')] Khovanov homology of a closed $4$-braid can have only $\mathbb{Z}_2$ and $\mathbb{Z}_4$ torsion.
		\item[(3)] Khovanov homology of a closed $n$-braid cannot have $p$-torsion for $p > n$ ($p$ prime).
		\item[(3')] Khovanov homology of a closed $n$-braid cannot have $\mathbb{Z}_{p^r}$ torsion for $p^r > n$.
	\end{enumerate}
\end{conjecture}

\begin{figure}
	\centering
	\includegraphics[width = 11cm]{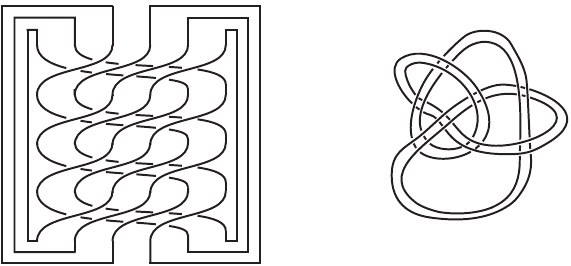}
	\caption{\small{The braid $(\sigma_4\sigma_3\sigma_2\sigma_1\sigma_5\sigma_4\sigma_3\sigma_2)^4$, the flat $2$-cabling of the knot $8_{19}$.}}
	\label{cabling}
\end{figure}

Observe that the flat $2$-cabling of the torus knot $T(2,2s+1)$ is equivalent to $T^{(-4s-2)}(4,4s+2)$.  From Theorem \ref{Theorem 4.1}, we have the following corollary.

\begin{corollary}\label{Corollary 3.14}
Suppose that $H^{6s+1,16s+4}(T^{(-4s-2)}(4,4s+2))$ has $Z_{2^s}$-torsion where $s \geq 1.$ Then for any $k \leq -4s-2,$ the Khovanov homology $H^{6s+1,20s+6+k}(T^{(k)}(4,4s+2))$ also contains $\mathbb{Z}_{2^s}$-torsion.
\end{corollary}

\begin{proof}
Notice that $w(T^{(k)}(4,4s+2))=3(4s+2)+k$ and $w(T^{(k)}_{B}(4,4s+2))=-(4s+2)+1-k,$ i.e. $u_{k}=2(4s+2)+k.$ Thus, the following isomorphism can be obtained from Theorem \ref{Theorem 4.1}:
$$H^{i,j}(T^{(k)}(4,4s+2))=H^{i,j-1}(T^{(k-1)}(4,4s+2))$$
if $(i,j) \not\in \{(u_{k},3u_{k}),(u_{k},3u_{k}-2),(u_{k}-1,3u_{k}),(u_{k}-1,3u_{k}-2)\}.$\\
We have $u_{k} \leq 4s+2 < 6s+1$ because $k \leq -4s-2$ and $s \geq 1,$ which means $u_{k}$ or $u_{k}-1$ never becomes $6s+1$ when $k \leq -4s-2.$ Therefore, since $H^{6s+1,16s+4}(T^{(-4s-2)}(4,4s+2))$ contains $Z_{2^s}$-torsion, so does $H^{6s+1,20s+6+k}(T^{(k)}(4,4s+2))$ for $k \leq -4s-2.$
\end{proof}

\begin{conjecture}\label{bigtorsions}
	Flat $2$-cabling of the torus knot $T(2,2s+1)$ has $\mathbb{Z}_{2^{s'}}$-torsion for $0 < s' \leq s$
	\begin{enumerate}
		\item in bidegree $(i,j),$ for $i = 1+8s-2s'$ and $j = 4 + 20s - 4s',$ where $s$ is a positive integer.
		\item  in bidegree $(i,j),$ for $i = 8s-2s'$ and $j = 20s - 4s',$ where $s$ is a positive integer greater than one.
	\end{enumerate}
	 
\end{conjecture}

The above conjecture is verified up to $s=23.$ Note that the knot $T^{(-2n-1)}(4,2n)$ and the link $T^{(-2n)}(4,2n)$ which are flat $2$-cablings of the torus link $T(2,n)$ have braid index $4$ for $n \geq 2.$ Therefore, they are counterexamples to Conjecture \ref{PS} $(2^{'})$ and the case of $p=2$ in Conjecture \ref{PS} $(3^{'})$. Tables \ref{T64m6}, \ref{T104m10}, and \ref{2-cabling_of_9_1_with_16_torsion} show the Khovanov homology of the links when $s=1,2,$ and $4$ in Conjecture \ref{bigtorsions}.

We end the section with the following problem.

\begin{problem}
Find knots or links whose Khovanov homology has $\mathbb{Z}_{p^s}$-torsion for odd prime $p$ and $s>1.$ In particular, 
\begin{enumerate}
	\item Determine whether or not the Khovanov homology of $T(9,10)$ and $T(9,11)$ has $\mathbb{Z}_9$-torsion.
	\item Find $\mathbb{Z}_9$-torsion for $4$-braid links.
\end{enumerate}
\end{problem}

\section{Torsion in reduced and odd Khovanov homology}\label{Section 5}

Even (reduced and unreduced) and odd (reduced and unreduced) Khovanov homology share the same long exact sequence of homology involving $D_B$, $D$, and $D_A$ \cite{ORS,Ras,Shu1,Shu2}. The reason for this is that all of these homology theories use the same decomposition of chain groups:
$$C_{a,b}(D) = C_{a,b}(D_{B-marker}) \oplus C_{a,b}(D_{A-marker})$$ (see Figure \ref{D-ABmarkers}), and $C_{**}(D_{B-marker})$ is a subchain complex of $C_{**}(D)$.
Boundary maps have the same bidegree and in each of the cases the exact sequence of Lemma \ref{Corollary 2.2.} holds. Thus we can have results
analogous to Corollaries \ref{T^{(k)}(4,6)}-\ref{T^{(k)}(7,9)}, \ref{Corollary 3.14} as long as we have initial data. We illustrate this by two examples,\ \\
(1) on reduced (even) Khovanov homology, where torsion is very rare.  Bidegrees of torsion different from $\mathbb{Z}_2$ in the reduced Khovanov homology of $T(7,9)$ is given in Remark \ref{Shu}.\ \\ 
(2) on odd Khovanov homology, where we use the reduced version since the unreduced version is just ``double" of the reduced one \cite{ORS}. For odd Khovanov homology, torsion is plentiful ($8_{19}=T(3,4)$ has $\mathbb{Z}_2$ and $\mathbb{Z}_3$-torsion), but only small examples have been computed \cite{Shu1}.\ \\
Initial data for the reduced odd Khovanov homology of $10_{124}=T(3,5)$ is given in Tables \ref{oT53}-\ref{oT53m2}. Thus we obtain the following theorem.

\begin{table}[h!]
	\centering
	\includegraphics[scale=0.6]{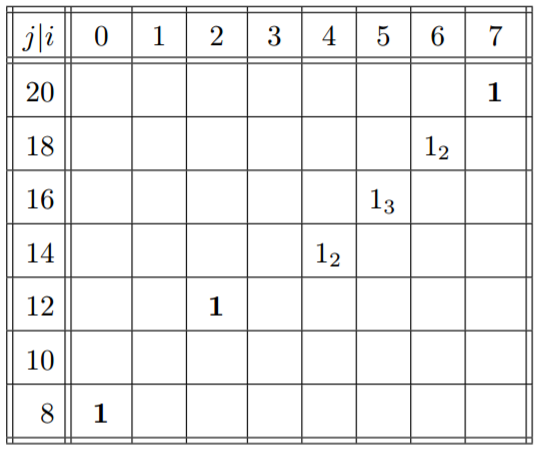}
	\caption{\small{Odd Khovanov homology of the torus knot $T(3,5)$.}}
	\label{oT53}
\end{table}

\begin{table}[h!]
	\centering
	\includegraphics[scale=0.6]{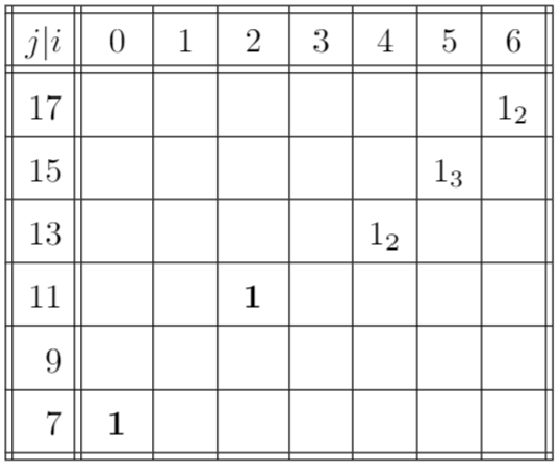}
	\caption{\small{Odd Khovanov homology of the link $T^{(-1)}(3,5)$.}}
	\label{oT53m1}
\end{table}

\begin{table}[h!]
	\centering
	\includegraphics[scale=0.6]{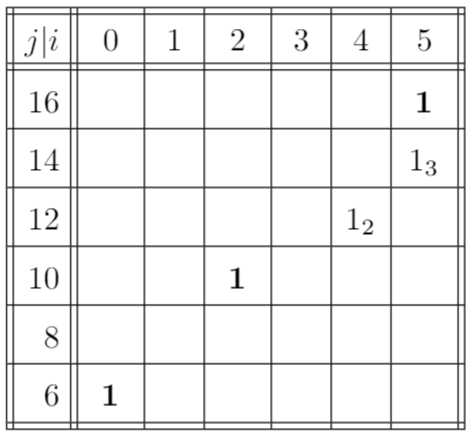}
	\caption{\small{Odd Khovanov homology of the knot $T^{(-2)}(3,5)$ which is also $T(4,5)$ and $8_{19}$.}}
	\label{oT53m2}
\end{table}

\begin{figure}[h!]
	\centering
	\includegraphics[scale=0.6]{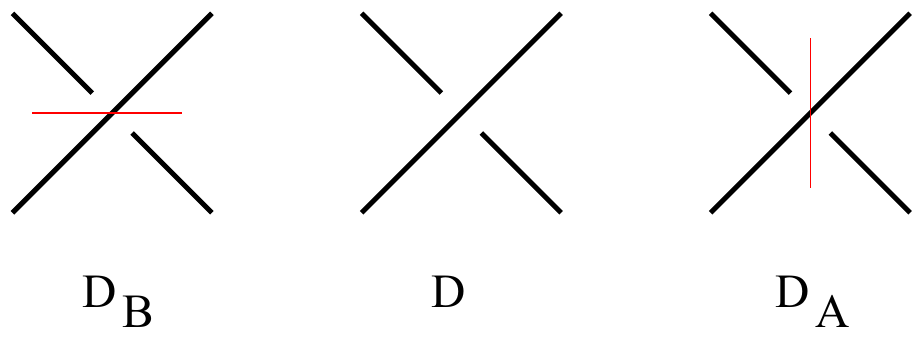} 
	\caption{\small{$D_{B-marker}, D$, and $D_{A-marker}$.}}
	\label{D-ABmarkers}
\end{figure}

\begin{theorem}\
\begin{enumerate}
\item[(1)] For all $k \in \mathbb{Z}$ the reduced (even) Khovanov homology of $T(7,9)$ satisfies
$$tor_{odd}\widetilde{H}^{17,7+k}(T^{(k)}(7,9))= {\mathbb Z}_5.$$ In particular, the knot $T^{(-8)}(7,9)$ of 46 crossings and the two component link $T^{-9}(7,9)$ of $45$ crossings have $\mathbb{Z}_{5}$-torsion.
\item[(2)] For any $k\geq -2$ the reduced odd Khovanov homology of $T(3,5)$ satisfies
$$tor \widetilde{H}_{odd}^{5,16+k}(T^{(k)}(3,5))= {\mathbb Z}_3.$$
As noted before, $T^{(-2)}(3,5)=T(3,4)=8_{19}$, see Proposition \ref{prop}.
\end{enumerate}
\end{theorem}

\begin{remark}\label{Shu}
	Shumakovitch provided us with data on the reduced (even) Khovanov homology of $T(7,9)$ \footnote{Shumakovitch informed us that $T(4,6)$ and $T(5,7)$ have only ${\mathbb Z}_2$-torsion and $T(6,8)$, only ${\mathbb Z}_2$ and ${\mathbb Z}_4$-torsion in reduced (even) Khovanov homology.}. We list them here in bidegrees which contain torsion different from $\mathbb{Z}_{2}$:
	$$\widetilde{H}^{17,7}(T(7,9))=\mathbb{Z}_2\oplus \mathbb{Z}_5,\ \widetilde{H}^{20,76}(T(7,9))=\mathbb{Z}_2^2\oplus \mathbb{Z}_5, \ \widetilde{H}^{22,78}(T(7,9))=\mathbb{Z}\oplus \mathbb{Z}_2^2\oplus \mathbb{Z}_5,$$
	$$\widetilde{H}^{25,84}(T(7,9))=\mathbb{Z}\oplus \mathbb{Z}_2\oplus \mathbb{Z}_5,$$
	$$\widetilde{H}^{26,84}(T(7,9))=\mathbb{Z}^2\oplus \mathbb{Z}_3, \ \widetilde{H}^{29,90}(T(7,9))=\mathbb{Z}\oplus \mathbb{Z}_3.$$
\end{remark}

\section*{Acknowledgements}
We would like to thank Lukas Lewark and Alexander Shumakovitch for providing initial computational data. The computational data for even Khovanov homology provided in this paper was obtained by using {\bf JavaKh}-v2 written by Scott Morrison, which is an update of Jeremy Green's {\bf JavaKh}-v1 written under the supervision of Dror Bar-Natan. Odd and reduced Khovanov homology was computed by using {\bf KhoHo}, written by Alexander Shumakovitch \cite{Shu4}. \\

J\'{o}zef H. Przytycki was partially supported by the Simons Foundation Collaboration Grant for Mathematicians-316446
and CCAS Dean's Research Chair award.
 Marithania Silvero was partially supported by MTM2016-76453-C2-1-P and FEDER.

{\tiny
DEPARTMENT OF MATHEMATICS, THE GEORGE WASHINGTON UNIVERSITY, WASHINGTON DC, USA.}\\
{\scriptsize {\it E-mail address}: sujoymukherjee@gwu.edu}

{\tiny
DEPARTMENT OF MATHEMATICS, THE GEORGE WASHINGTON UNIVERSITY, WASHINGTON DC, USA, AND UNIVERSITY OF GDA\'{N}SK, POLAND.}\\
{\scriptsize{\it E-mail address}: przytyck@gwu.edu}

{\tiny
DEPARTAMENTO DE \'{A}LGEBRA, UNIVERSIDAD DE SEVILLA, SPAIN, AND INSTITUTE OF MATHEMATICS OF THE POLISH ACADEMY OF SCIENCES, WARSAW, POLAND.} \\
{\scriptsize{\it E-mail address}: marithania@us.es}

{\tiny
DEPARTMENT OF MATHEMATICS, THE GEORGE WASHINGTON UNIVERSITY, WASHINGTON DC, USA.}\\
{\scriptsize {\it E-mail address}: wangxiao@gwu.edu}

{\tiny
DEPARTMENT OF MATHEMATICS, THE GEORGE WASHINGTON UNIVERSITY, WASHINGTON DC, USA.}\\
{\scriptsize {\it E-mail address}: syyang@gwu.edu}

\pagebreak

\section*{Appendix}\label{APP}

\begin{table}[h!]
	\centering
	\includegraphics[scale=0.6]{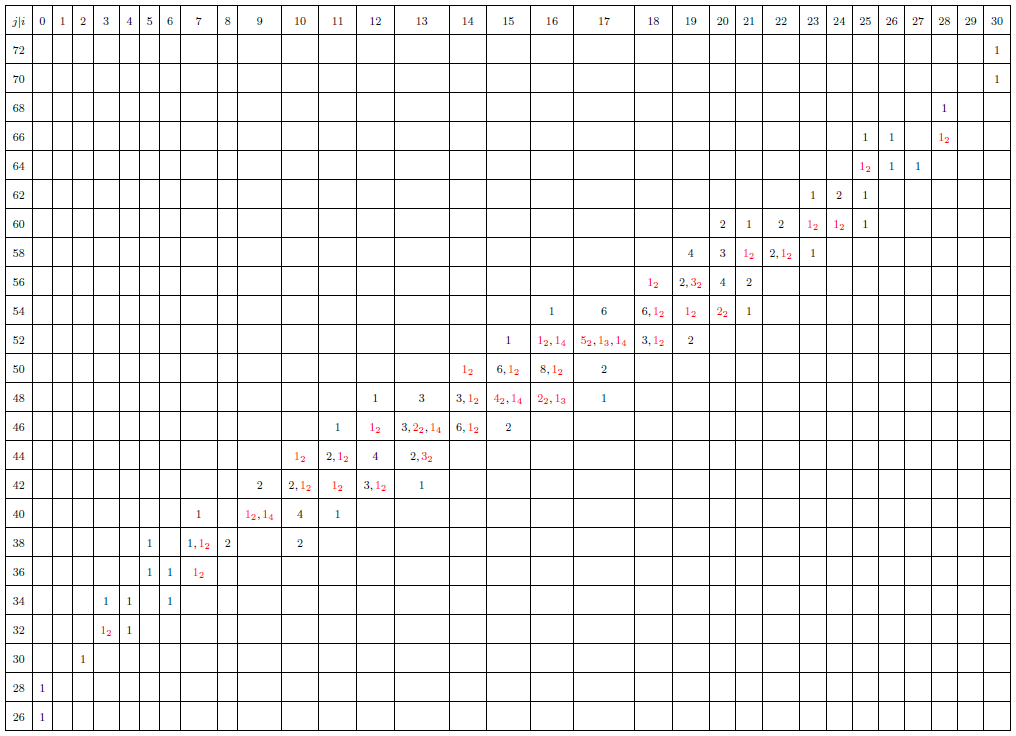}
	\caption{\small{Khovanov homology of the closure of the braid $(\sigma_3\sigma_2\sigma_1\sigma_1\sigma_2\sigma_3)^5.$}}
	\label{4b4}
\end{table}

\begin{table}[h!]
	\centering
	\includegraphics[scale=0.6]{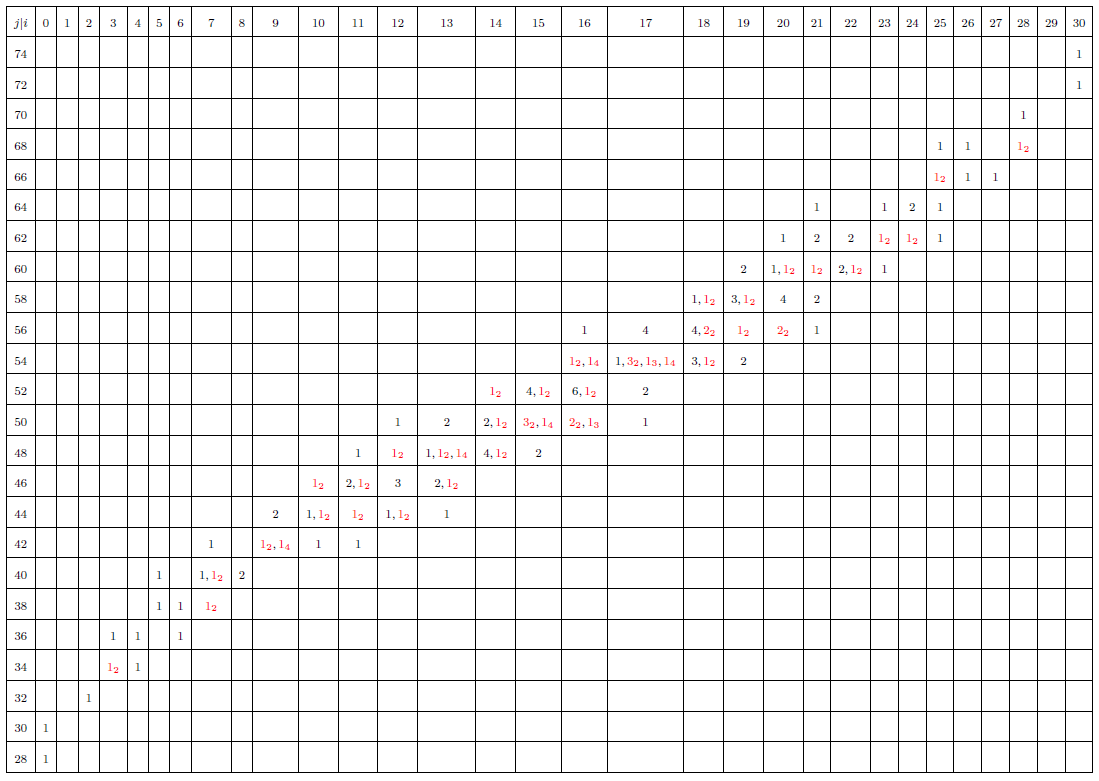}
	\caption{\small{Khovanov homology of the closure of the braid $(\sigma_3\sigma_2\sigma_1\sigma_1\sigma_2\sigma_3)^5\sigma_2\sigma_1.$}}
	\label{4b2}
\end{table}

\pagebreak

\begin{table}[h!]
	\centering
	\includegraphics[scale=0.6]{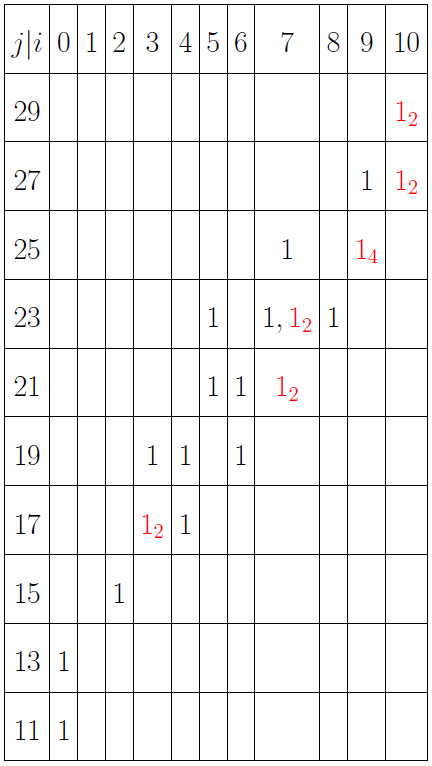}
	\caption{\small{Khovanov homology of the torus knot $T(4,5)$.}}
	\label{T54}
\end{table}

\begin{table}[h!]
	\centering
	\includegraphics[scale=0.6]{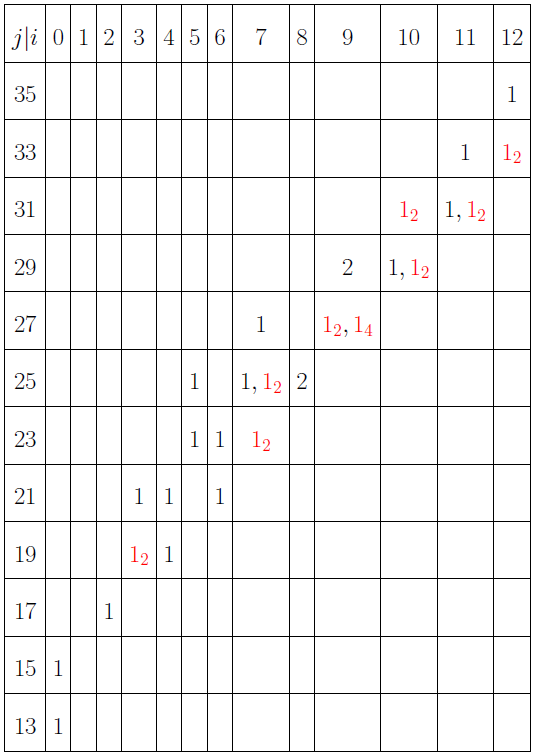}
	\caption{\small{Khovanov homology of the knot $T^{(2)}(4,5)$.}}
	\label{T542}
\end{table}

\pagebreak

\begin{table}[h!]
\centering
\includegraphics[scale=0.6]{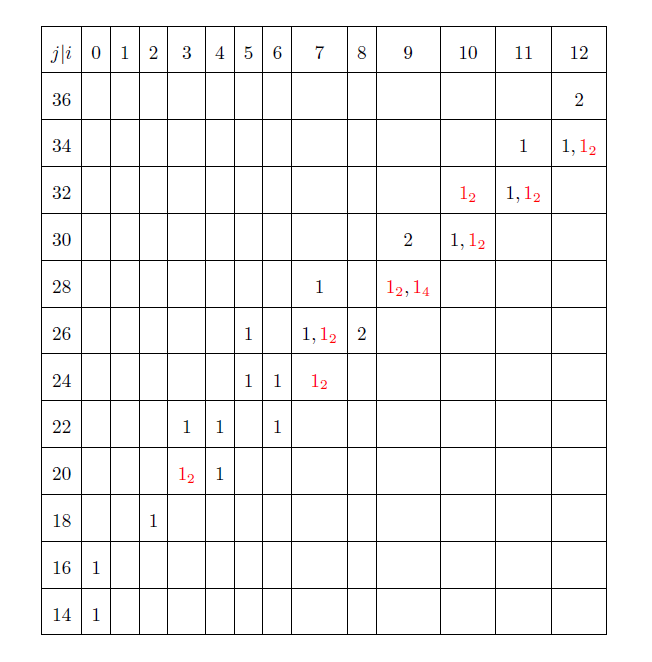}
\caption{\small{Khovanov homology of the torus link $T(4,6)$.}}
\label{T64}
\end{table}

\begin{table}[h!]
	\centering
	\includegraphics[scale=0.6]{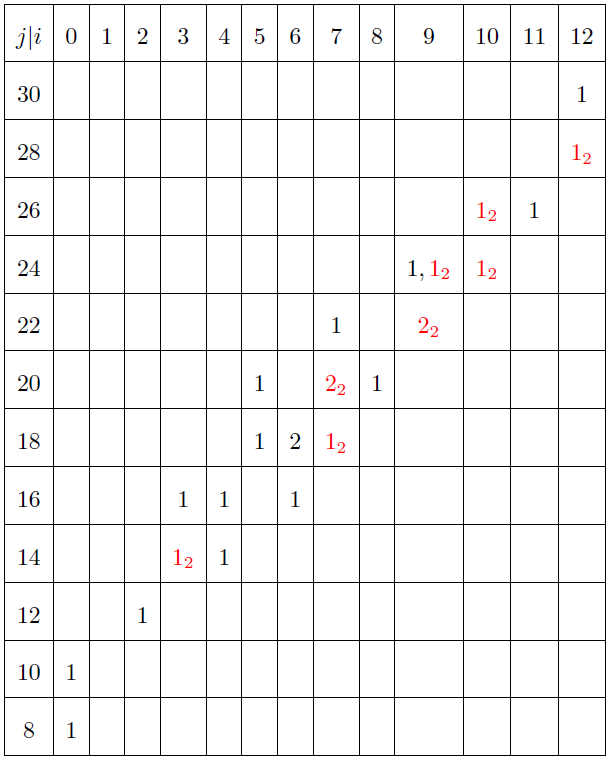}
	\caption{\small{Khovanov homology of the link $T^{(-6)}(4,6)$.}}
	\label{T64m6}
\end{table}

\pagebreak

\begin{table}[h!]
\centering
\includegraphics[scale=0.6]{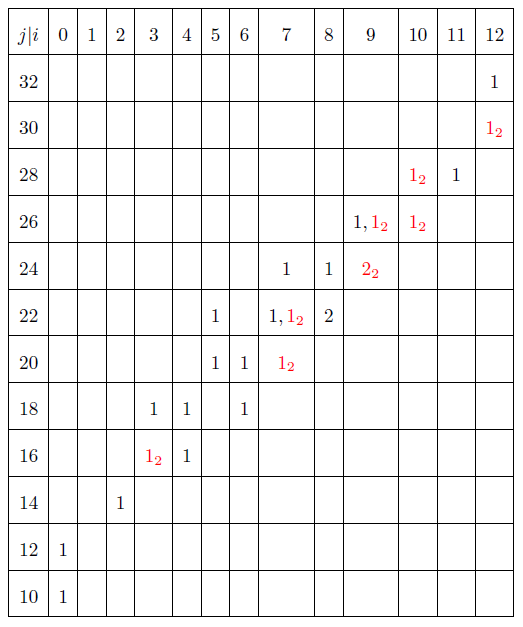}
\caption{\small{Khovanov homology of the link $T^{(-4)}(4,6)$.}}
\label{T64m4}
\end{table}

\begin{table}[h!]
\centering
\includegraphics[scale=0.6]{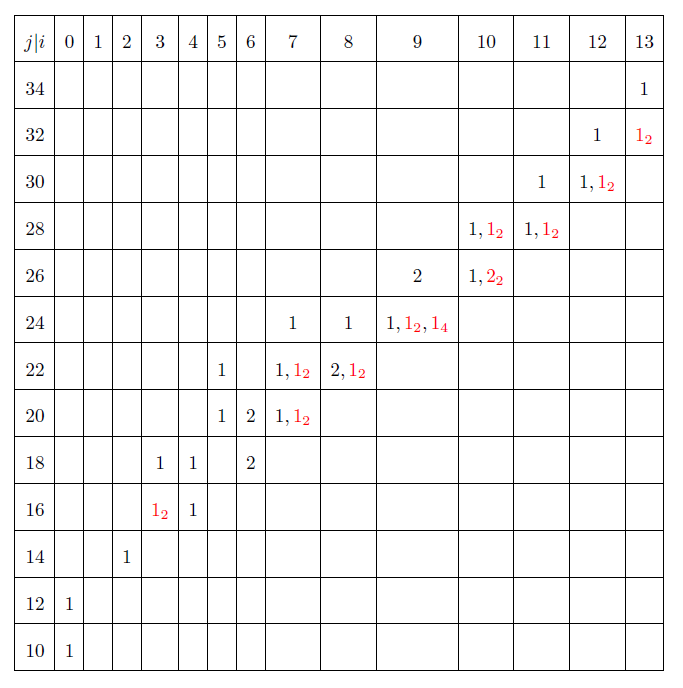}
\caption{\small{Khovanov homology of the closure of the braid $(\sigma_3\sigma_2\sigma_1)^7\sigma_1^{-5}\sigma_3^{-2}.$ After reduction, it has a diagram of $14$ crossings.}}
\label{T74m5_with-1-1}
\end{table}

\pagebreak

\begin{table}[h!]
	\centering
	\includegraphics[scale=0.6]{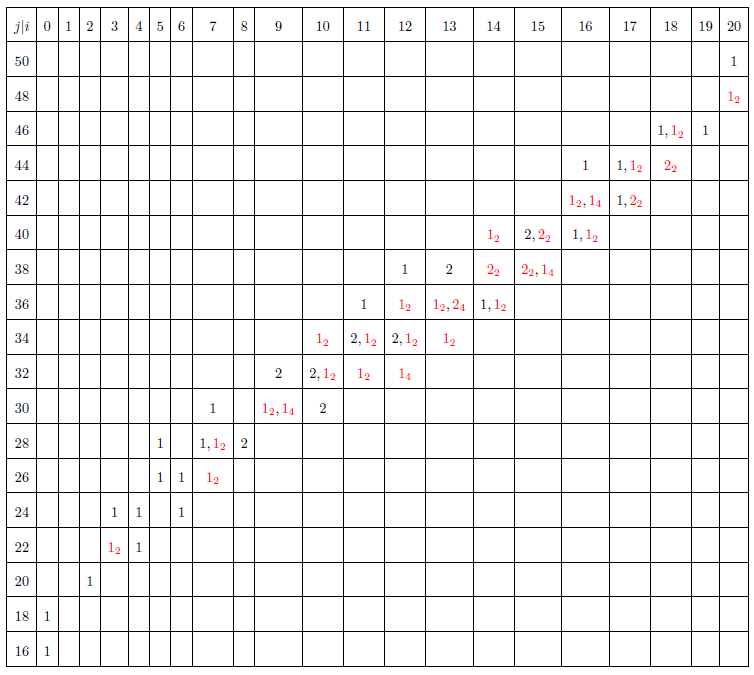}
	\caption{\small{Khovanov homology of the link $T^{(-10)}(4,10)$.}}
	\label{T104m10}
\end{table}

\begin{table}[h!]
	\centering
	\includegraphics[scale=0.6]{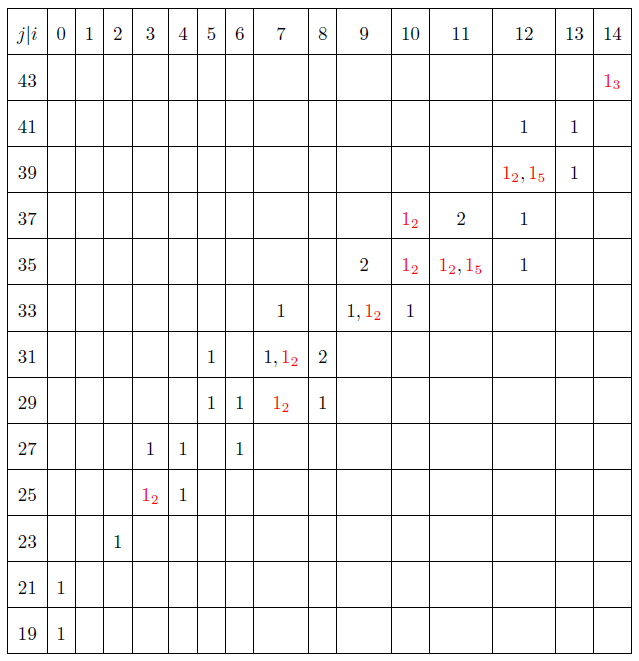}
	\caption{\small{Khovanov homology of the torus knot $T(5,6)$.}}
	\label{T65}
\end{table}

\pagebreak

\begin{table}[h!]
	\centering
	\includegraphics[scale=0.6]{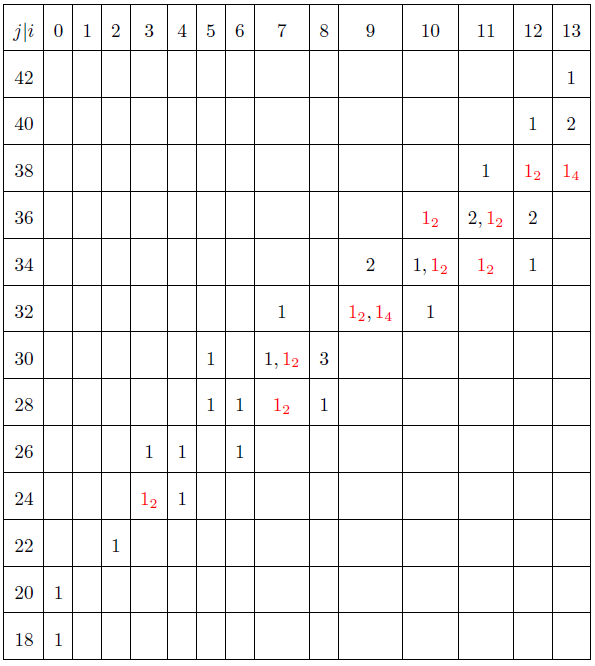}
	\caption{\small{Khovanov homology of the link $T^{(-1)}(5,6)$.}}
	\label{T65m1}
\end{table}

\begin{table}[h!]
	\centering
	\includegraphics[scale=0.6]{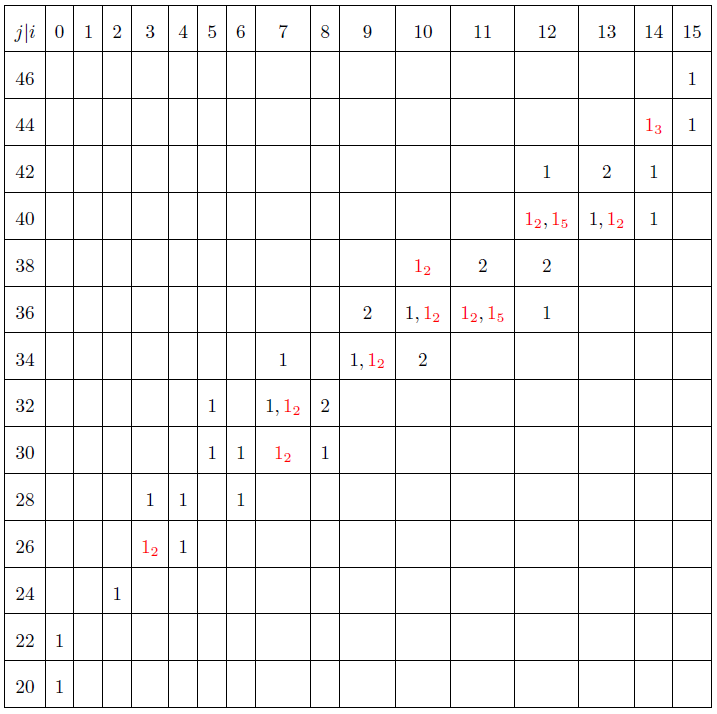}
	\caption{\small{Khovanov homology of the link $T^{(1)}(5,6)$.}}
	\label{T651}
\end{table}

\pagebreak

\begin{table}[h!]
	\centering
	\includegraphics[scale=0.6]{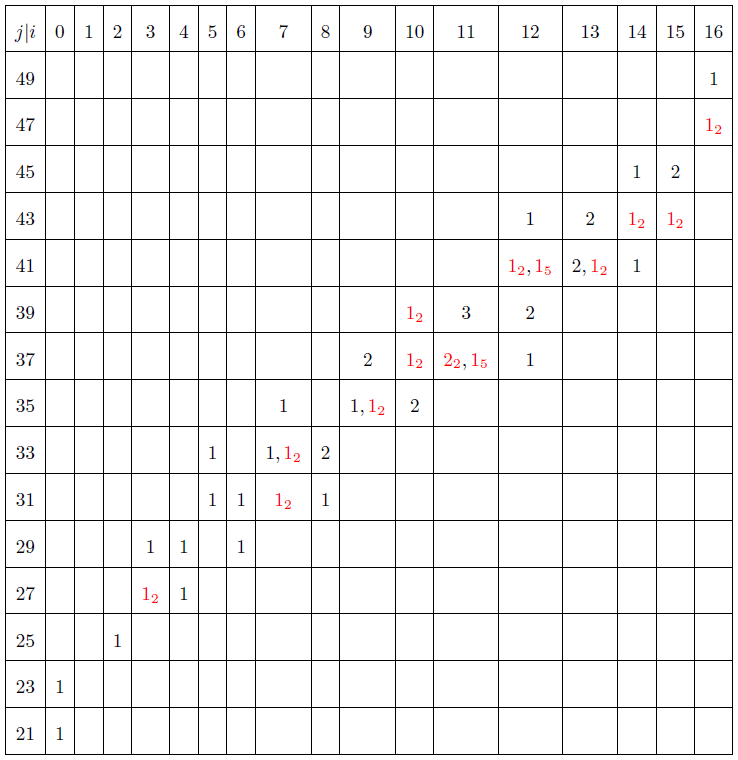}
	\caption{\small{Khovanov homology of the knot $T^{(2)}(5,6)$.}}
	\label{T652}
\end{table}

\begin{table}[h!]
	\centering
	\includegraphics[scale=0.6]{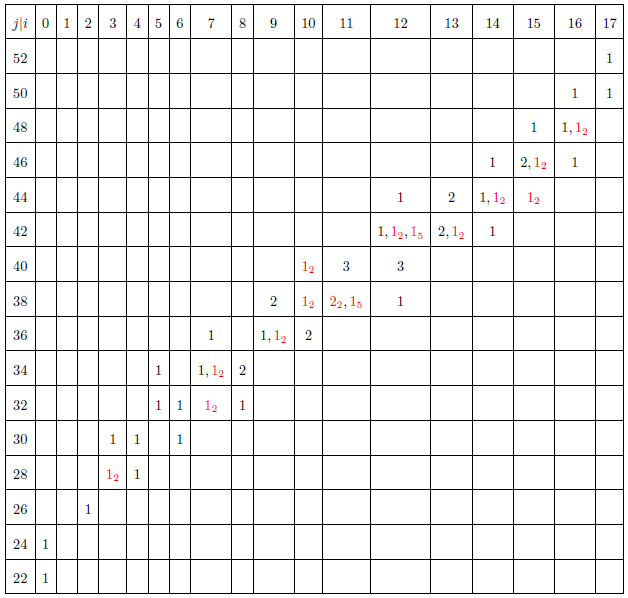}
	\caption{\small{Khovanov homology of the link $T^{(3)}(5,6)$.}}
	\label{T653}
\end{table}

\pagebreak

\begin{table}
	\centering
	\includegraphics[scale=0.6]{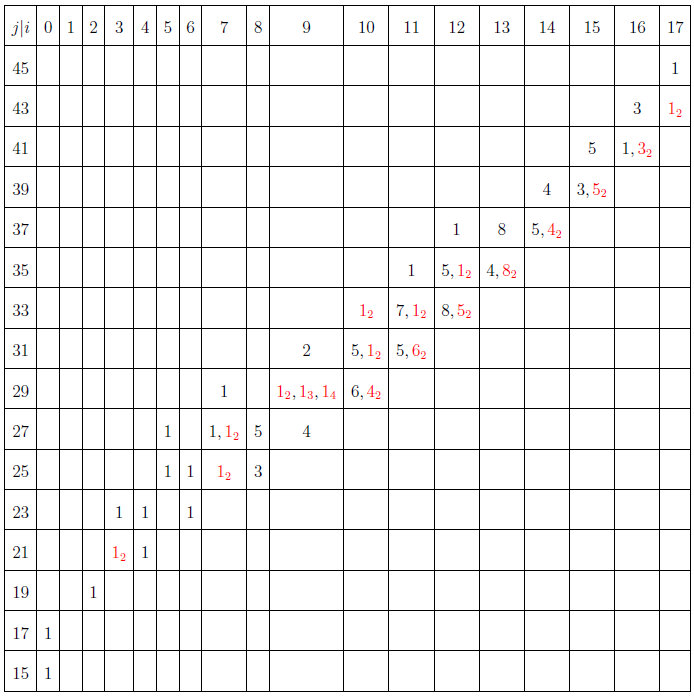}
	\caption{\small{Khovanov homology of the closure of the braid $(\sigma_3\sigma_2\sigma_1\sigma_4\sigma_3)^4$.}}
	\label{32143_to_the power_4}
\end{table}

\begin{table}
	\centering
	\includegraphics[scale=0.6]{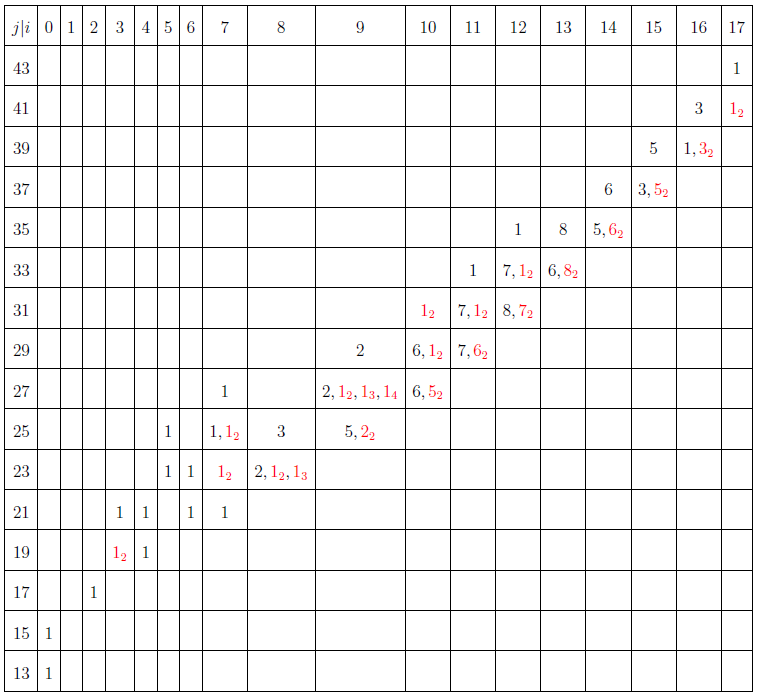}
	\caption{\small{Khovanov homology of the closure of the braid $(\sigma_3\sigma_2\sigma_1\sigma_4\sigma_3)^4\sigma_2^{-1}\sigma_4^{-1}.$}}
	\label{32143_to_the power_4_-3-1}
\end{table}

\begin{table}
\centering
\includegraphics[scale=0.6]{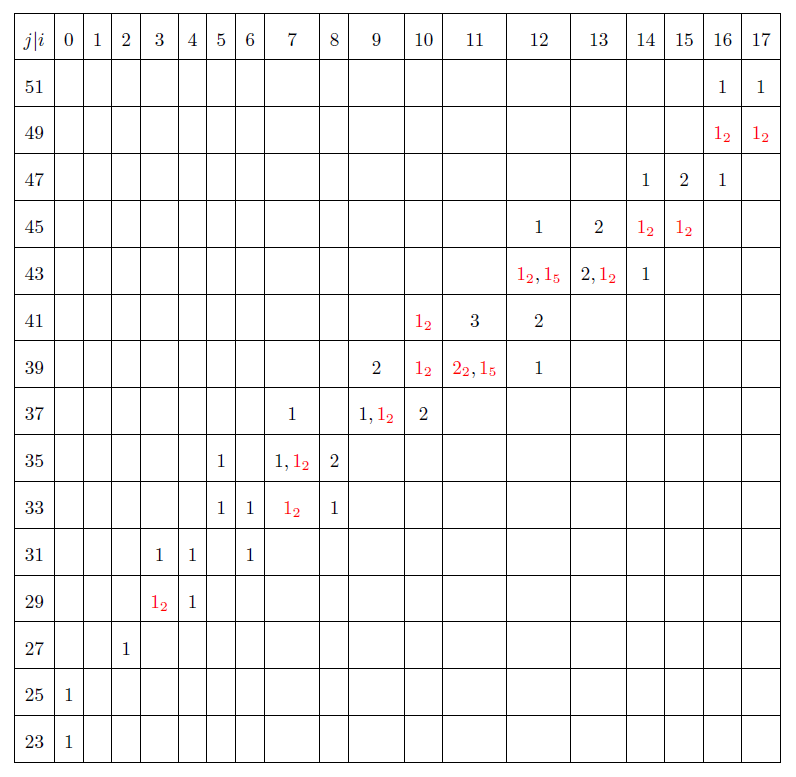}
\caption{\small{Khovanov homology of the torus knot $T(5,7)$.}}
\label{T75}
\end{table}

\begin{table}
\centering
\includegraphics[scale=0.6]{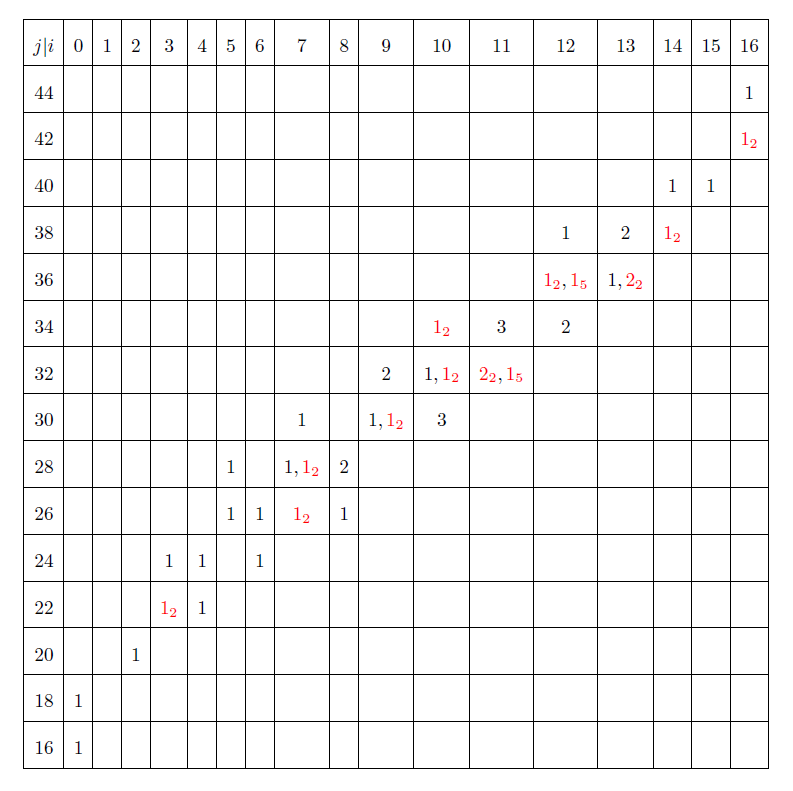}
\caption{\small{Khovanov homology of the link $T^{(-7)}(5,7)$.}}
\label{T75m7}
\end{table}

\begin{table}
\centering
\includegraphics[scale=0.6]{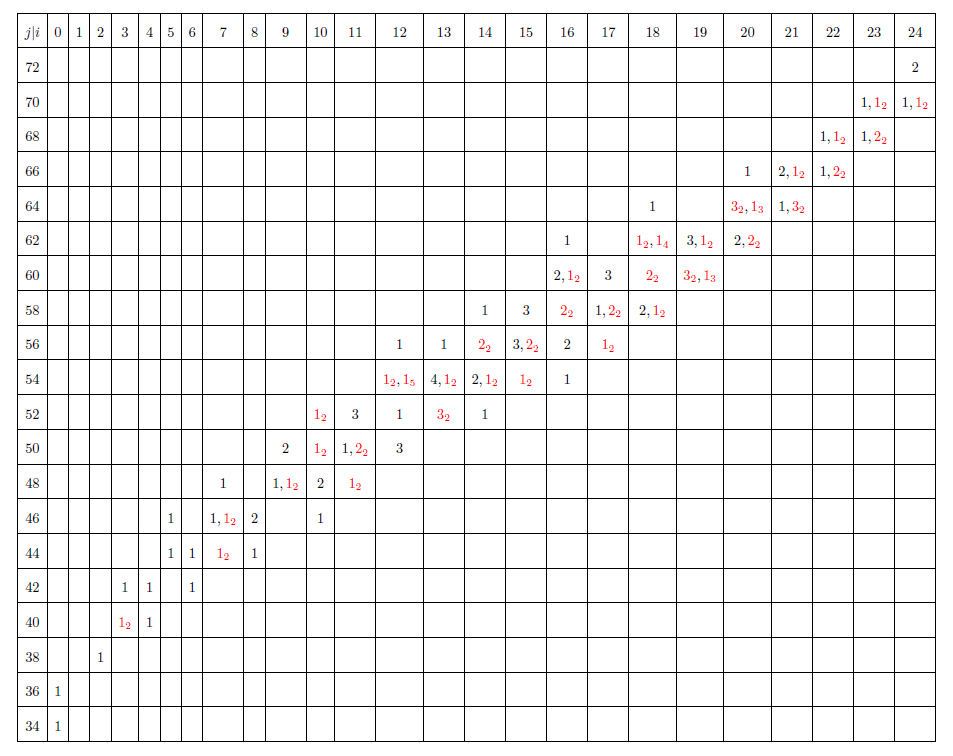}
\caption{\small{Khovanov homology of the torus link $T(6,8)$.}}
\label{T86}
\end{table}

\begin{table}
\centering
\includegraphics[scale=0.6]{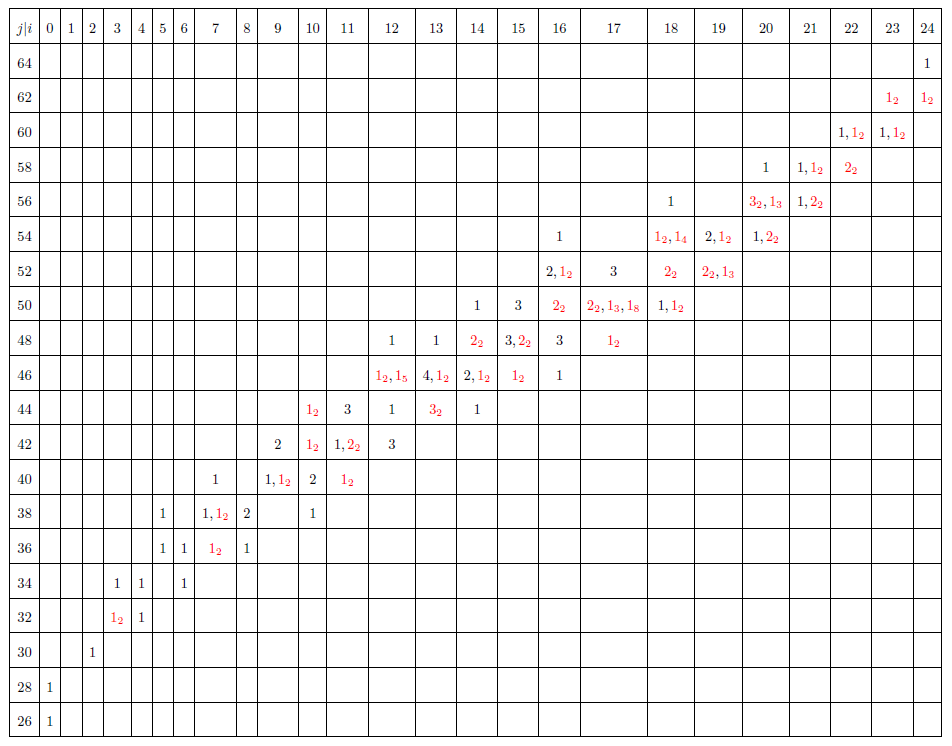}
\caption{\small{Khovanov homology of the link $T^{(-8)}(6,8)$.}}
\label{T86m8}
\end{table}

\begin{table}
	\centering
	\includegraphics[scale=0.6]{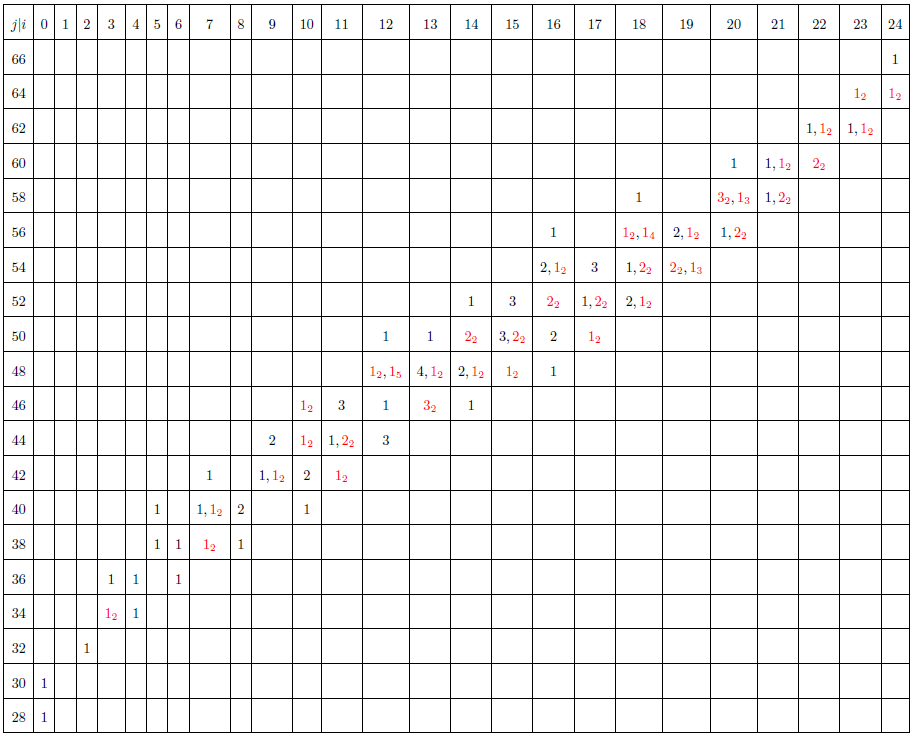}
	\caption{\small{Khovanov homology of the link $T^{(-6)}(6,8)$.}}
	\label{T86m6}
\end{table}

\begin{table}
	\centering
	\includegraphics[scale=0.6]{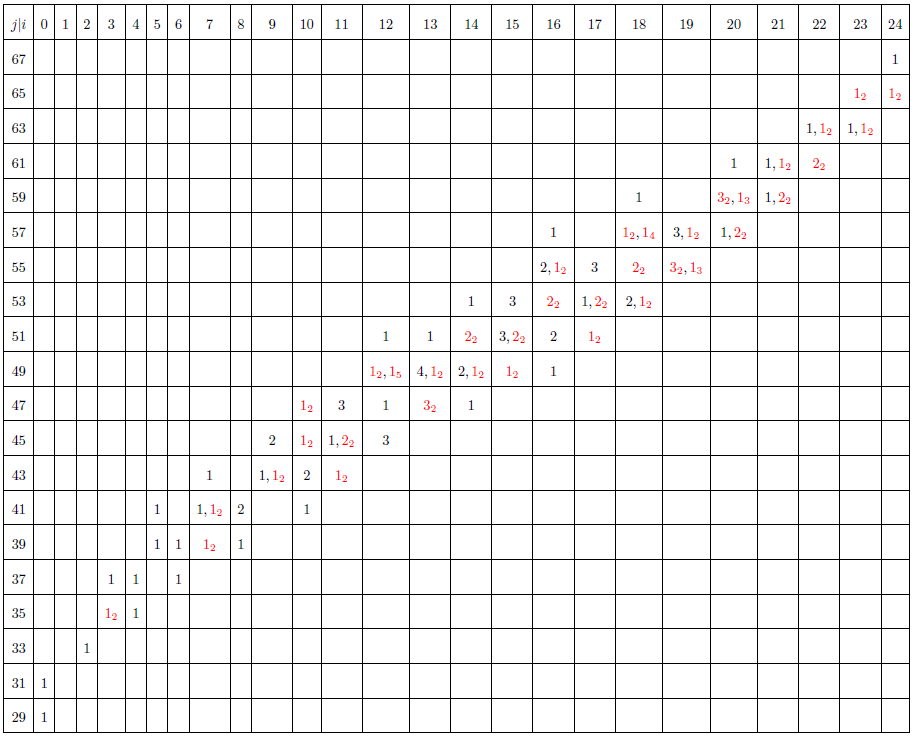}
	\caption{\small{Khovanov homology of the knot $T^{(-5)}(6,8)$.}}
	\label{T86m5}
\end{table}

\begin{table}
\centering
\includegraphics[scale=0.6]{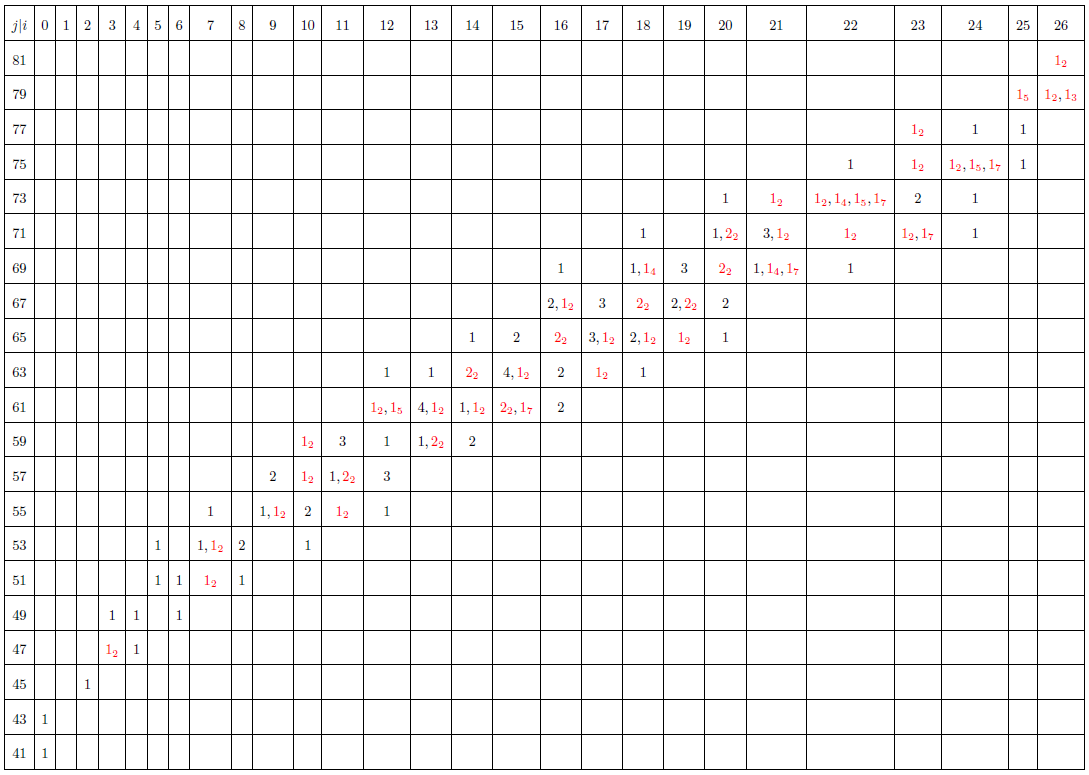} 
\caption{\small{Khovanov homology of the torus knot $T(7,8)$.}}
\label{T87}
\end{table}

\begin{table}
	\centering
	\includegraphics[scale=0.6]{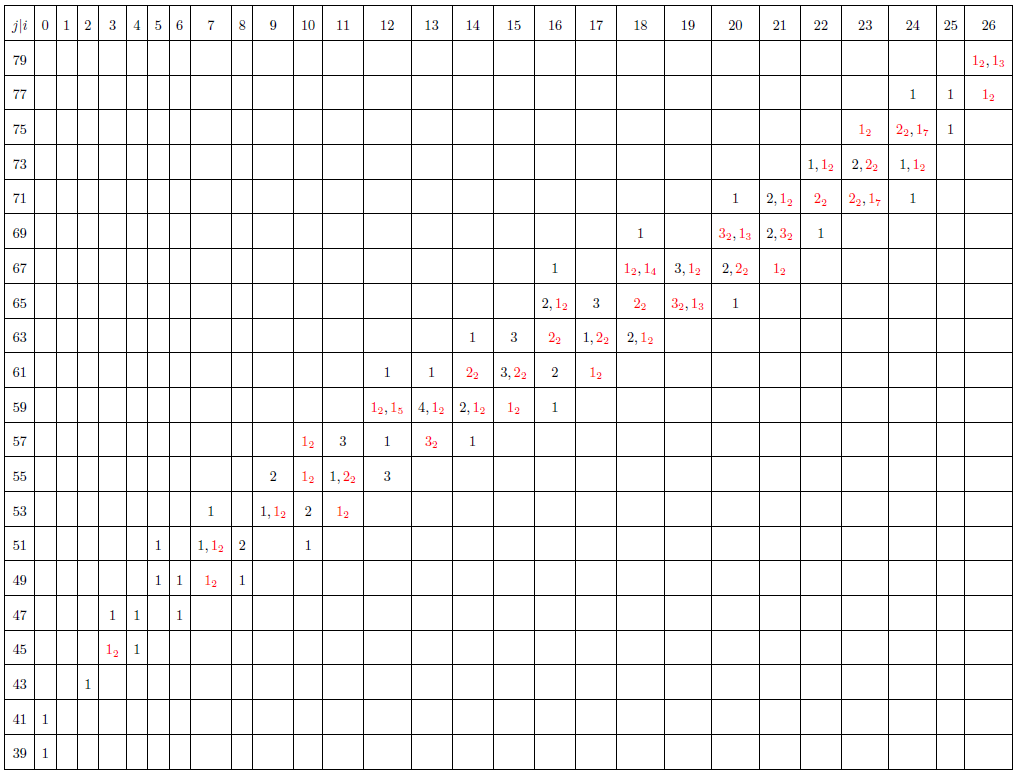} 
	\caption{\small{Khovanov homology of the knot $T^{(-2)}(7,8)$.}}
	\label{T87m2}
\end{table}

\begin{table}
	\centering
	\includegraphics[scale=0.6]{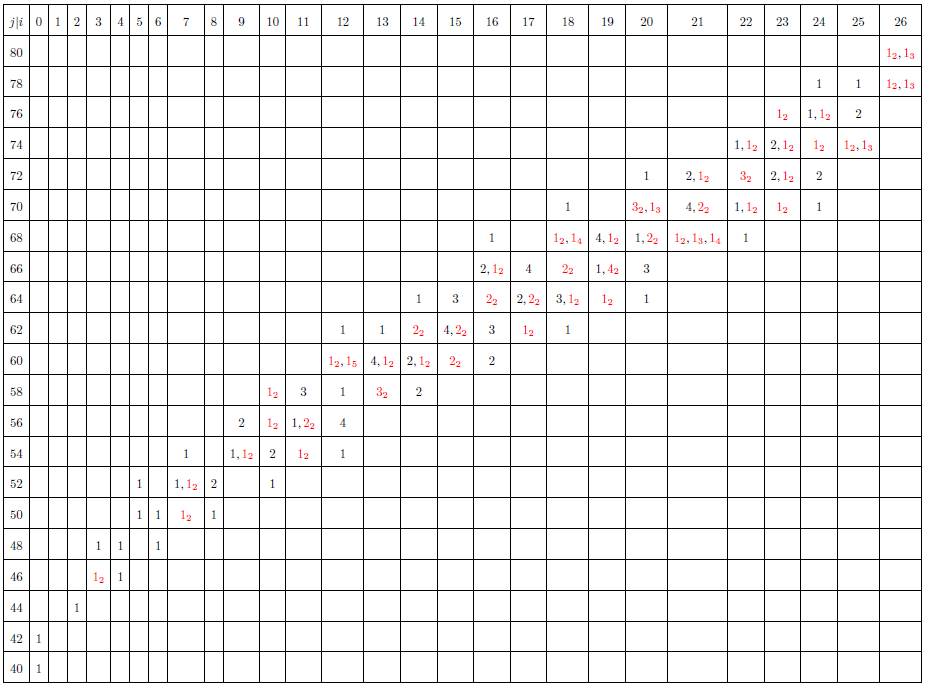} 
	\caption{\small{Khovanov homology of the link $T^{(-1)}(7,8)$.}}
	\label{T87m1}
\end{table}

\begin{table}
	\centering
	\includegraphics[scale=0.6]{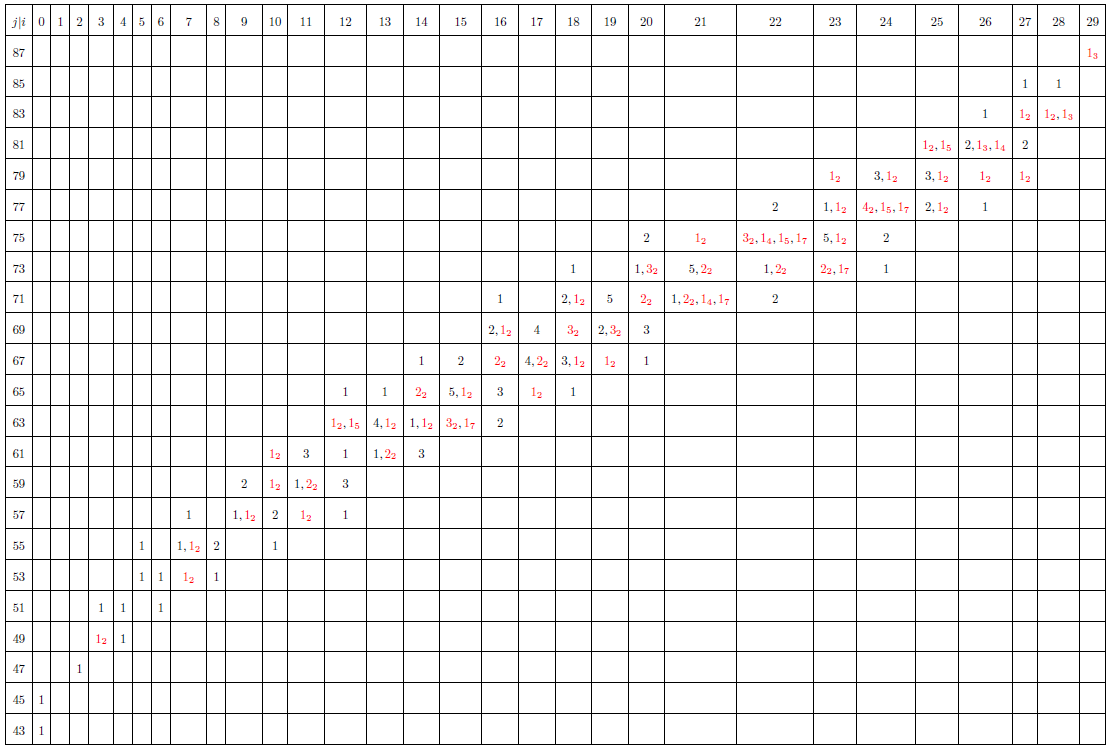} 
	\caption{\small{Khovanov homology of the knot $T^{(2)}(7,8)$.}}
	\label{T872}
\end{table}

\begin{table}
\centering
\includegraphics[scale=0.6]{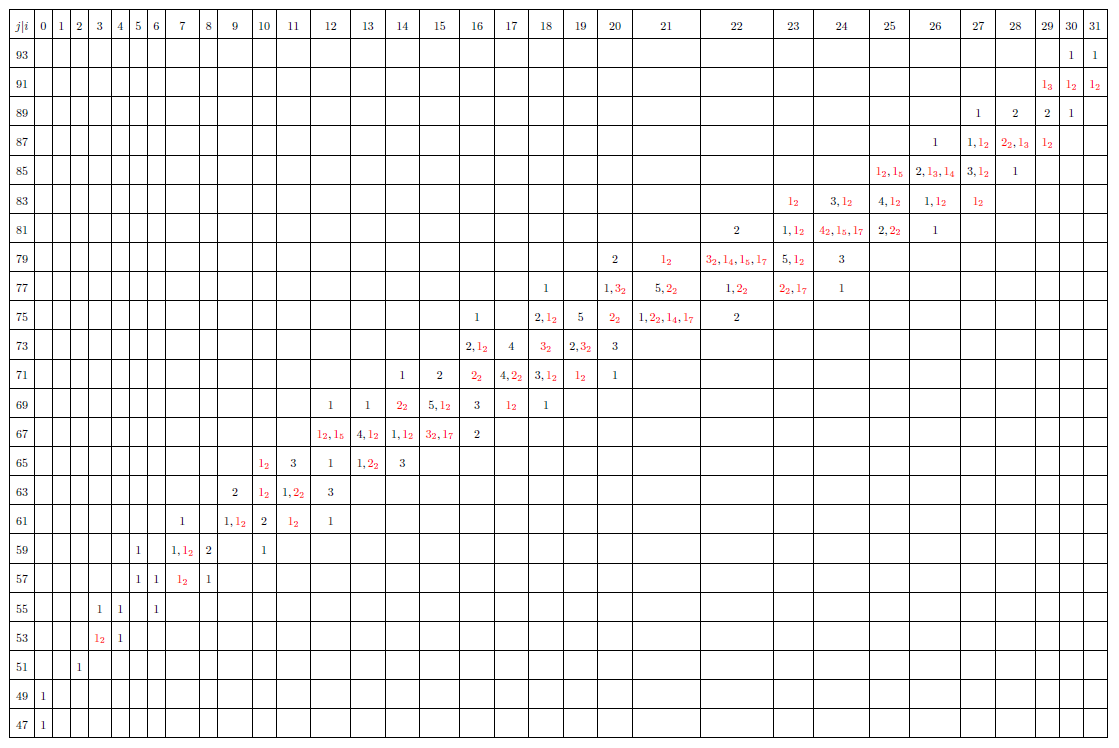}
\caption{\small{Khovanov homology of the torus knot $T(7,9)$.}}
\label{T97}
\end{table}

\begin{table}
\centering
\includegraphics[scale=0.6]{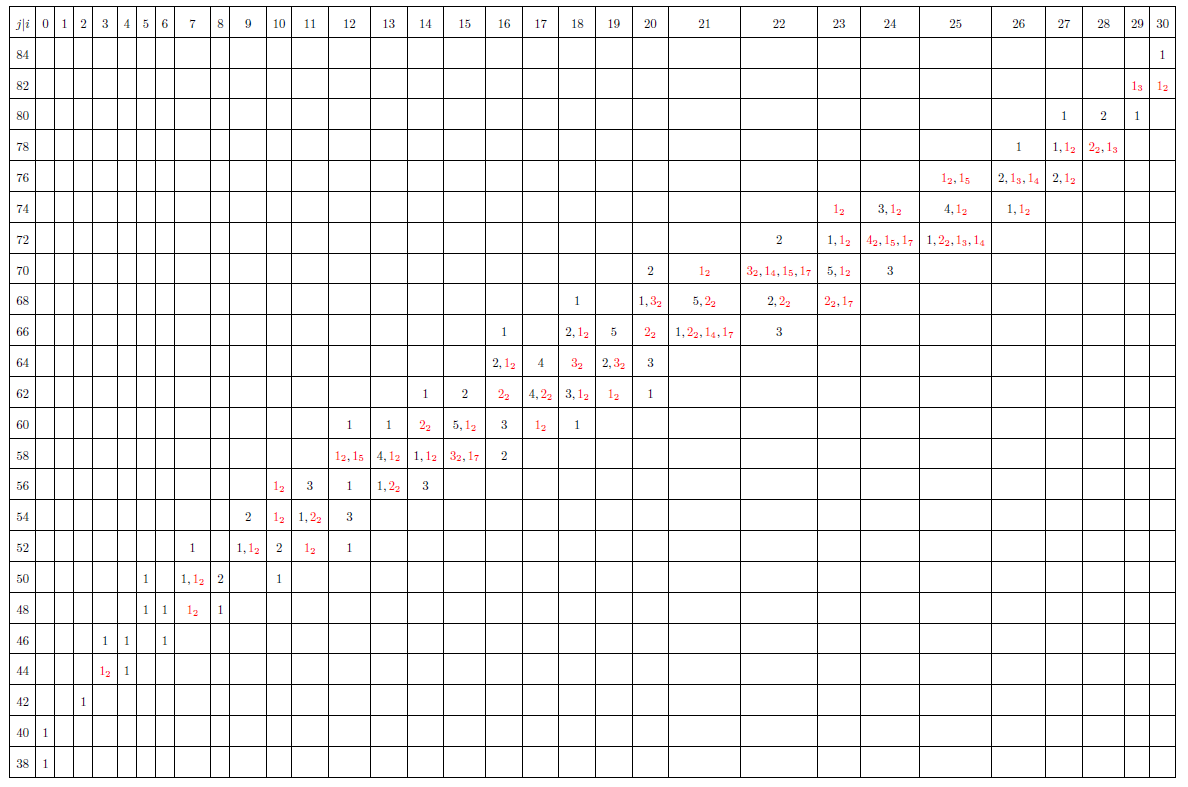}
\caption{\small{Khovanov homology of the link $T^{(-9)}(7,9)$.}}
\label{T97m9}
\end{table}

\begin{table}
	\centering
	\includegraphics[scale=0.6]{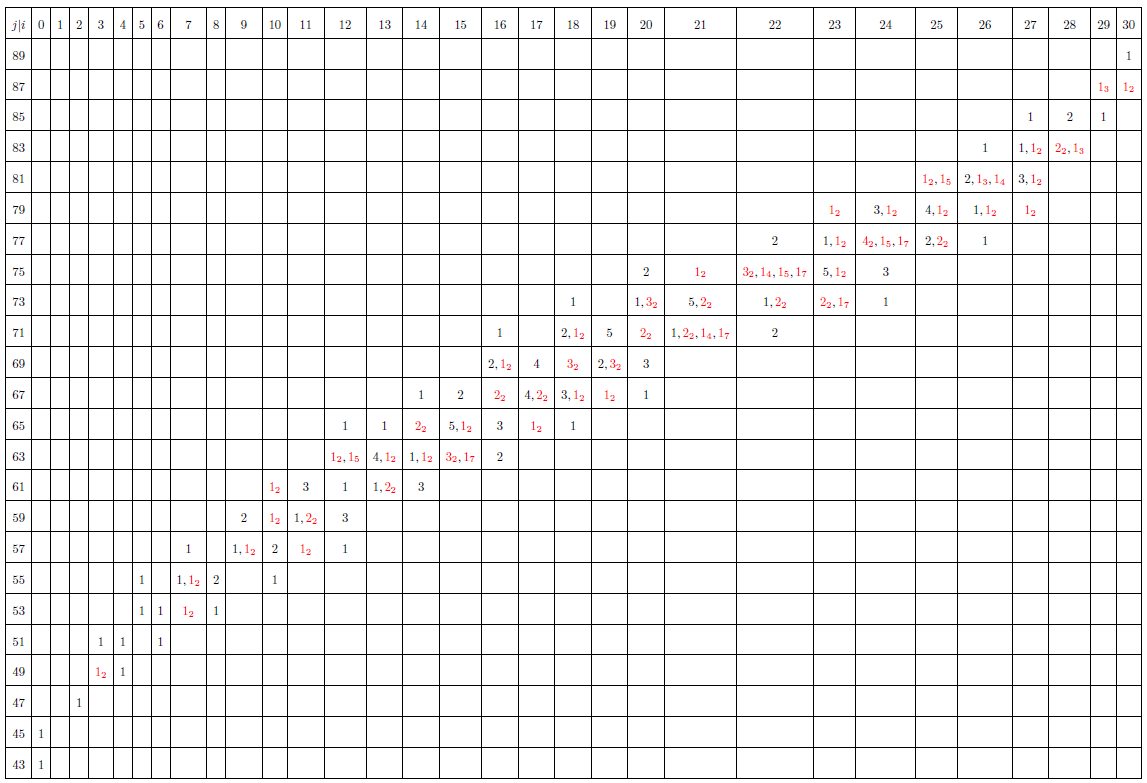}
	\caption{\small{Khovanov homology of the knot $T^{(-4)}(7,9)$.}}
	\label{T97m4}
\end{table}

\pagebreak

\begin{table}
	\centering
	\includegraphics[scale=0.97]{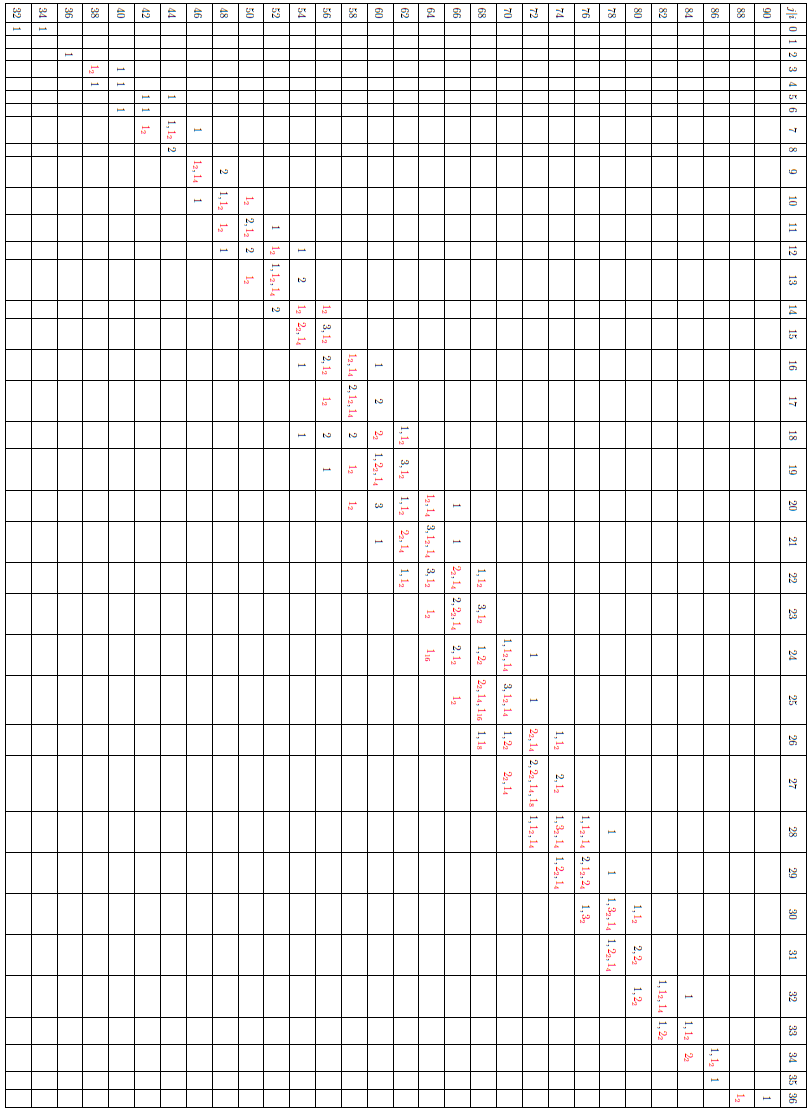}
	\caption{\small{Khovanov homology of the flat $2$-cabling of the knot $9_1$.}}
	\label{2-cabling_of_9_1_with_16_torsion}
\end{table}

\end{document}